\documentclass[a4paper, 12pt]{article}

\usepackage[sort&compress]{natbib}
\bibpunct{(}{)}{;}{a}{}{,} 

\usepackage{amsthm, amsmath, amssymb, mathrsfs, multirow, url, subfigure}
\usepackage{graphicx} 
\usepackage{ifthen} 
\usepackage{amsfonts}
\usepackage[usenames]{color}
\usepackage{fullpage}
\usepackage{tikz}

\theoremstyle{plain} 
\newtheorem{thm}{Theorem}
\newtheorem{cor}{Corollary}
\newtheorem{prop}{Proposition}

\theoremstyle{definition}
\newtheorem{defn}{Definition}

\theoremstyle{remark} 

\newtheorem{remark}{Remark}

\newcommand{\prob}{\mathsf{P}} 
\newcommand{\E}{\mathsf{E}}

\newcommand{\unif}{{\sf Unif}}
\newcommand{\nm}{{\sf N}}

\newcommand{\A}{\mathcal{A}} 
\renewcommand{\AA}{\mathbb{A}}
 
\newcommand{\RR}{\mathbb{R}}

\newcommand{\OO}{\mathbb{O}}
\newcommand{\XX}{\mathbb{X}}
\newcommand{\YY}{\mathbb{Y}}
\newcommand{\ZZ}{\mathbb{Z}}

\newcommand{\TT}{\mathbb{T}}

\renewcommand{\SS}{\mathbb{S}}

\newcommand{\eps}{\varepsilon}

\newcommand{\model}{\mathscr{P}}
\newcommand{\prior}{\mathsf{Q}}
\newcommand{\credal}{\mathscr{Q}}

\newcommand{\pbcred}{\mathscr{R}}
\newcommand{\pb}{\mathsf{R}}

\newcommand{\lPi}{\underline{\Pi}}
\newcommand{\uPi}{\overline{\Pi}}

\newcommand{\lPhi}{\underline{\Phi}}
\newcommand{\uPhi}{\overline{\Phi}}

\newcommand{\lprior}{\underline{\mathsf{Q}}}
\newcommand{\uprior}{\overline{\mathsf{Q}}}

\newcommand{\uuprob}{\mathbf{\overline{P}}}

\newcommand{\eval}{\mathfrak{e}}

\title{Regularized e-processes: anytime valid inference with knowledge-based efficiency gains}
\author{Ryan Martin\footnote{Department of Statistics, North Carolina State University, {\tt rgmarti3@ncsu.edu}}
}
\date{\today}

\begin{document}

\maketitle 

\begin{abstract}
Classical statistical methods have theoretical justification when the sample size is predetermined.  In applications, however, it's often the case that sample sizes are data-dependent rather than predetermined.  The aforementioned methods aren't reliable in this latter case, hence the recent interest in e-processes and methods that are anytime valid, i.e., reliable for any dynamic data-collection plan.  But if the investigator has relevant-yet-incomplete prior information about the quantity of interest, then there's an opportunity for efficiency gain.  This paper proposes a {\em regularized e-process} framework featuring a knowledge-based, imprecise-probabilistic regularization with improved efficiency.  A generalized version of Ville's inequality is established, ensuring that inference based on the regularized e-process are anytime valid in a novel, knowledge-dependent sense.  Regularized e-processes also facilitate possibility-theoretic uncertainty quantification with strong frequentist-like calibration properties and other Bayesian-like properties: satisfies the likelihood principle, avoids sure-loss, and offers formal decision-making with reliability guarantees. 

\smallskip

\end{abstract}


\section{Introduction}
\label{S:intro}

Most statistical and machine learning methods in the literature have theoretical justification that assumes the sample size is (or can safely be treated as) fixed, say, by the experimental or data-collection protocol.  But in many applications this assumption is violated---e.g., decisions to stop/continue experimentation are made dynamically while monitoring the data---yet the same fixed-sample-size methods are used for data analysis.  Use of methods when they lack theoretical justification raises doubts about their reliability.  This lack of reliability is surely at least a contributor to the widely-publicized replication crisis in science \citep[e.g.,][]{nuzzo2014, baker.nature.2016, camerer2018}.  

To address this concern, there's been a surge of effort to develop {\em safe} or {\em anytime valid} statistical methods, i.e., methods that can preserve their reliability even when applied in cases where the data-collection process is dynamic; see the recent survey by \citet{evalues.review} and Section~\ref{SS:eprocess} below.  At the heart of these new developments are {\em e-processes}, which satisfy a time-uniform probability bound, namely, {\em Ville's inequality}, that can be leveraged to construct procedures with error rate control guarantees that hold no matter how the data-collection process is stopped.  
Different anytime valid methods are typically compared in terms of their efficiency, i.e., power of tests, size of confidence intervals, etc.  In terms of the e-process itself, efficiency corresponds to how quickly it grows when evaluated at the ``wrong'' hypotheses: fast growth of the e-process is desirable because it means that what's ``wrong'' can be detected sooner, hence practitioners can make safe and reliable decisions with less time/exposure and fewer resources.   

But e-processes have limits to how fast they can grow \citep[e.g.,][]{grunwald.safetest}.  What can be done to improve efficiency when there are no more e-process tweaks to be made? It's now statistical second-nature to leverage penalty functions or prior distributions for {\em regularization}, to encourage certain structure in estimates, usually for the sake of efficiency gain.  Regularization is relatively easy when the goal is asymptotic consistency or calibration, since many different regularization strategies work in that sense. But anytime validity is a finite-sample property, so virtually any tweak made to an e-process would jeopardize the anytime validity property that motivated its use in the first place.  

This paper considers the situation in which the investigator {\em knows something} about $\Theta$, the relevant quantity of interest, before the data are available for analysis.  To keep the names and roles clear, ``I'' refers to me, this paper's author, and ``You'' refers to the investigator, a catch-all term for the individual(s) carrying out a study, analyzing data, etc.  The simple phrase {\em knows something} is rather nuanced.  On the one hand, I mean {\em know} in the strong sense of absolute certainty whereas, on the other hand, the {\em something} that's known could be so substantial that it renders data irrelevant, so miniscule that it means practically nothing, or somewhere in the middle of these two extremes.  
According to \citet{levi1980}, Your {\em corpus of knowledge} consists of the (lower and upper) probabilities that You'd assign to hypotheses about the uncertain $\Theta$ based on the information available to You at the relevant moment in time.  
The point I want to emphasize is that Your corpus of knowledge is what it is at the time You call on it; without new information that warrants revising Your corpus, there's no justification for adding to or subtracting from it. Modern data science, however, focuses on methods that can be applied off-the-shelf which, by design, are incapable of accommodating Your prior knowledge and, therefore, expect You to subtract from Your corpus of knowledge. 
The present paper shows how You can use precisely what You know and retain statistical reliability.  

The key insight is leveraging Your corpus of knowledge in two new and distinct ways.  The first is {\em regularization}, which amounts to directly manipulating the data-driven e-process using what You know.  Unless You're ignorant, Your corpus of knowledge casts a degree of doubt on some values $\theta$ of $\Theta$.  The regularizer appropriately discounts those $\theta$ values, thereby boosting the regularized e-process there.  A bigger e-process means more $\theta$ values can be excluded from consideration, hence more efficiency.  The second is introducing a generalized, {\em regularization-aware} notion of anytime validity with respect to which the regularized e-process is evaluated.  Since You can't doubt what You know, You'd be willing to evaluate the performance of Your regularized e-process based on a metric that depends on what You know.  Roughly, my proposal requires that the regularized e-process be anytime valid with respect to any $(\text{data}, \Theta)$-joint distribution compatible with both the data-generating model and Your corpus of knowledge., the latter being expressed in terms of a credal set of ``priors'' for $\Theta$.  The standard definition of anytime validity corresponds to the special---and most restrictive---case where the aforementioned credal set contains all priors for $\Theta$.  By moving away from this most-restrictive case, regularization allows for efficiency gains without ruining anytime validity. 

The chief technical development in the paper is the establishment of a generalized version of Ville's inequality for my proposed regularized e-process, one that depends in a certain way on Your corpus of knowledge.  This leads to provably reliable, non-asymptotic inference with (or without) regularization and opportunities for efficiency gain.  Beyond construction of test and confidence procedures, a full-blown framework for (possibilistic) uncertainty quantification, which facilitates both calibrated, data-dependent (imprecise) probabilistic reasoning about $\Theta$, and reliable, data-driven decision-making.  


The organization of the paper is as follows.  After some background in Section~\ref{S:background} on e-processes and imprecise probability, Section~\ref{S:reg.eprocess} incorporates Your corpus of knowledge into the data-driven, e-process-based analysis via regularization,   
establishes an imprecise-probabilistic generalization of Ville's inequality, and investigates its statistical implications.  Illustrations of the possible efficiency gains are presented in Section~\ref{SS:gains}.  Focus shifts in Section~\ref{S:reg.uq} to broader uncertainty quantification about $\Theta$.  What I'm proposing is a brand of {\em inferential model} (IM), where the familiar probabilistic reasoning is replaced by provably reliable possibilistic reasoning. My proposed e-possibilistic IM framework inherits the anytime validity from the regularized e-process, which implies that my uncertainty quantification is calibrated in a strong sense, hence reliable.  Remarkably, in addition to these frequentist-like calibration properties, the e-possibilistic IM satisfies several other desirable Bayesian-like properties: satisfying the likelihood principle, avoiding sure-loss, and formal decision-making with strong reliability guarantees.  An application analyzing clinical trial data is presented in Section~\ref{S:ware}, illustrating how subject matter knowledge gets translated into a regularizer.  Concluding remarks are made in Section~\ref{S:discuss}, and further technical details are given in the supplementary materials.




\section{Background}
\label{S:background}

\subsection{Setup and notation}
\label{SS:setup}

Start with a baseline probability space $(\SS, \A, \prob)$, where $\A$ is a $\sigma$-algebra of subsets of $\SS$ and $\prob$ is a probability measure.  Let $Z: \SS \to \ZZ$ be a measurable function that takes values in a topological space $\ZZ$.  
I'll also write $\prob$ for the induced distribution of $Z$.  

Next, let $Z_1, Z_2, \ldots$ denote independent and identically distributed (iid) copies of $Z$ and, for each $n \geq 1$, write $Z^n = (Z_1,\ldots,Z_n)$.  
Define the filtration $\A_n$ to be the sequence of $\sigma$-algebras determined by the information in  $Z^n$.  A stopping time $N$ is a positive integer-valued random variable such that, for each $n$, the event $\{N=n\}$ is in $\A_n$.  


For the statistical applications that I have in mind, $\prob$ is {\em uncertain}.  
To facilitate this notion of an ``uncertain $\prob$,'' it'll help to introduce (the notation of) a model, namely, $\model = \{\prob_\omega: \omega \in \OO\}$, indexed by $\OO$.  Of course, this could be a familiar parametric model, but it could also be that $\OO$ is in one-to-one correspondence with the set of all relevant probability distributions, so there's no loss of generality in introducing the index $\omega$.  

The goal then is to quantify uncertainty about the uncertain $\prob$ or, equivalently, about the uncertain index $\Omega$, based on observations $Z_1,Z_2,\ldots$ from the underlying process that depends on $\prob$ or $\Omega$.  It's often the case that only certain features of the uncertain ($\prob$ or) $\Omega$ are relevant to the application at hand---e.g., maybe one only needs to know about the mean survival time of patients---so I'll define this relevant feature as $\Theta = f(\Omega)$, taking values in the possibility space $\TT = f(\OO)$.  
Since $f$ could be the identity function, in which case $\Theta = \Omega$, this focus on quantifying uncertainty about $\Theta$ based on data $Z_1,Z_2,\ldots$ is without loss of generality.  

One mild technical condition I'll impose is that the topology on $\OO$ is sufficiently rich that {\em $f$ is continuous}.  Continuity isn't necessary for the developments here (see Remark~\ref{re:continuity} in Appendix~\ref{A:remarks}), but it's no serious practical constraint and it greatly simplifies the developments in Section~\ref{SS:pullback}.  
For parametric statistical models, the relevant quantity would either be the parameter itself, or some interpretable feature thereof, e.g., one component of a parameter vector, so the function is almost always trivially continuous.  When no parametric model is assumed, as is often the case in machine learning applications, the relevant features $\Theta = f(\Omega)$ are often risk minimizers.  Let $L_t: \ZZ \to \RR$ be a loss function indexed by $t \in \TT$ and define the mapping $(\omega, t) \mapsto \prob_\omega L_t$, the expected loss corresponding to $t$ at $\omega$.  If this mapping is continuous, then Berge's {\em maximum theorem} \citep[e.g.,][Theorem~17.31]{infinite.dim.analysis} implies that $f(\omega) := \arg\min_t \prob_\omega L_t$ is continuous.


\subsection{E-processes}
\label{SS:eprocess}

Start by fixing a particular $\omega \in \OO$.  A sequence $(M^n: n \geq 0)$ is a {\em supermartingale}, relative to $\prob_\omega$ and the filtration $(\A_n)$, if $\E_\omega(M^n \mid \A_{n-1}) \leq M^{n-1}$, where $\E_\omega$ denotes expected value with respect to $\prob_\omega$.  I have in mind a function $M(\cdot)$ that maps $\ZZ$-valued sequences to numbers, and $M^n = M(Z^n)$ for each $n$, with $M^0 = M(\square)$ the value of $M$ when applied to the empty sequence $\square$.  An $\omega$-dependent supermartingale $(M_\omega^n: n \geq 0)$ is a {\em test supermartingale} for $\omega$ if it's non-negative and $M_\omega^0 \equiv 1$ \citep[e.g.,][]{shafer.vovk.martingale}.  A sequence $(\eval_\omega^n: n \geq 0)$ is an {\em e-process} for $\omega$ if it's non-negative and upper bounded by a test supermartingale for $\omega$; see \citet{ramdas.nnm} and \citet{evalues.review} for details.  Again, $(\eval_\omega^n)$ is determined by a mapping $\eval_\omega(\cdot)$, i.e., $\eval_\omega^n = \eval_\omega(Z^n)$.  Test supermartingales and, hence, e-processes satisfy two key properties: the first \citep[e.g.,][Theorem~5.7.6]{durrett2010} is 
\[ \E_\omega ( M_\omega^N ) \leq 1 \quad \text{all $\omega \in \OO$, all stopping times $N$}, \]
and the second, known as {\em Ville's inequality} \citep[e.g.,][]{shafer.vovk.book.2019}, is
\[ \prob_\omega(M_\omega^N \geq \alpha^{-1}) \leq \alpha \quad \text{all $\alpha \in (0,1]$, all stopping times $N$}. \]
Results of this type have important statistical implications \citep[e.g.,][]{shafer.betting}.  One is that, if $(\eval_\omega^n)$ is an e-process for $\omega$, then the test that rejects the hypothesis ``$\Omega=\omega$'' based on data $Z^n$ if and only if $\eval_\omega(Z^n) \geq \alpha^{-1}$ controls the frequentist Type~I error at level $\alpha$, regardless of what stopping rule might be used to terminate the data-collection process.  This {\em anytime validity}, or {\em safety} \citep{grunwald.safetest}, of the e-process-based tests is a major advancement beyond the classical tests that are valid only for fixed $n$. Of course, if one has a collection of e-processes $(\eval_\omega^n)$, one for each $\omega \in \OO$, then the above testing procedure can be inverted to construct a confidence set for $\Omega$ which inherits the same anytime validity property: the frequentist coverage probability is no less than the nominal level independent of the choice of stopping rule. 

With one exception, my examples below use Savage--Dickey e-processes, i.e., Bayes factors that rely on ``default'' priors.  Such constructions are quite general, but difficulties can arise in more complex problems.  My proposal described in the following sections can be applied to any e-process construction, including those in \citet{evalues.review}.  

A practically important extension of the ideas presented above is to the case of composite hypotheses.  Even the simple-looking hypotheses I have in mind here are composite---that is, ``$\Theta=\theta$'' generally corresponds to a set of $\omega$ values for $\Omega$.
Fortunately, there's an easy way to accommodate this more general case under certain measurability constraints: if  $(\eval_\omega^n)$ is an e-process for each $\omega \in \OO$, then, with a minor abuse of notation, the following is an e-process for $\theta$,
\[ \eval_\theta^n = \inf_{\omega \in \OO: f(\omega) = \theta} \eval_\omega^n, \quad \theta \in \TT. \]
Ignoring issues concerning the potential non-measurability of the above infima,
it follows immediately from the properties discussed above that 
\begin{equation}
\label{eq:eval.bound}
\sup_{\omega: f(\omega) = \theta} \E_\omega \{\eval_\theta(Z^N) \} \leq 1 \quad \text{all $\theta \in \TT$, all stopping times $N$}. 
\end{equation}
There's also a corresponding version of Ville's inequality:
\begin{equation}
\label{eq:ville}
\sup_{\omega \in \OO: f(\omega) = \theta} \prob_\omega\bigl\{ \eval_\theta(Z^N) \geq \alpha^{-1} \bigr\} \leq \alpha \quad \text{all $\alpha \in (0,1]$, all stopping times $N$}. 
\end{equation}
I should mention that constructing a e-process for $\Omega$ and then marginalizing to $\Theta$ via optimization as just described is not the only way.  Indeed, one can construct e-processes for $\Theta$ directly, bypassing $\Omega$ altogether; see \citet{dey.gue} and Section~\ref{SS:gue} below. 

\subsection{Imprecise probability}
\label{SS:ip}


The most familiar approach to quantifying uncertainty about $\Theta$ is to introduce a probability measure $\prior$, supported on subsets of $\TT$ and, then, for any relevant hypothesis ``$\Theta \in H$,'' where $H \subseteq \TT$, Your uncertainty about its truthfulness is quantified by $\prior(H)$, the $\prior$-probability of $H$. (Henceforth I'll refer to both $H$ and ``$\Theta \in H$'' as hypotheses about $\Theta$.) But it's easy to imagine having limited information and, hence, You can't precisely state the probability of $H$ for each hypothesis $H$.  This includes the extreme case of {\em vacuous} information, where You literally know nothing about $\Theta$, so all You can say is that $\prior(H) \in [0,1]$ for all $H \not\in \{\varnothing, \TT\}$.  More generally, You might be in between the two extremes of knowing nothing and knowing $\prior$ precisely.  For example, suppose ``You're 95\% sure that $\Theta$ is bigger than 7.''  This information imposes bounds on $\prior(H)$ for some $H$, e.g., $\prior(H) \leq 0.05$ for all $H \subseteq (-\infty, 7]$, which are satisfied by many different $\prior$, so this information fails to identify a single $\prior$.  This ambiguity alone isn't grounds for You to ignore the information altogether.

When the information available about $\Theta$ determines a single $\prior$, then Your uncertainty quantification is {\em precise}; otherwise, it's {\em imprecise}.  Mathematically, the latter case can be handled using {\em imprecise probabilities}.  I must emphasize that precise probability isn't superior to imprecise probability.  Your information about $\Theta$ is what it is, so artificially embellishing on that information to avoid addressing the inherent imprecision doesn't make Your uncertainty quantification better.  The goal is to represent Your corpus of knowledge as faithfully as possible; if that requires imprecision, then so be it. 

A fairly general way to define imprecise probabilities is via a {\em credal set} $\credal$, i.e., a non-empty, closed, and convex collection of probabilities $\prior$ on $\TT$ \citep[e.g.,][Sec.~4.2]{levi1980}.  Intuitively, $\credal$ encodes Your knowledge about $\Theta$, so the collection $\credal(H) = \{ \prior(H): \prior \in \credal \}$ of credences quantifies uncertainty about the truthfulness of $H$.  Natural summaries of this collection include the lower and upper bounds:
\begin{equation}
\label{eq:lower.upper}
\lprior(H) = \inf\{ \prior(H): \prior \in \credal \} \quad \text{and} \quad \uprior(H) = \sup\{ \prior(H): \prior \in \credal \}. 
\end{equation}
Thanks to the structure in $\credal$, the lower and upper probabilities are linked together, 
i.e., $\lprior(H) = 1 - \uprior(H^c)$ for all $H$.  For instance, if ``You're 95\% sure that $\Theta$ exceeds 7'' case above, then You get $\lprior(H)=0$ and $\uprior(H) = 0.05$ for all $H \subseteq (-\infty, 7]$, $\lprior(H)=0.95$ for all $H$ with $H \subseteq (7,\infty)$, and $\uprior(H) = 1$ for all $H$ with $H \cap (7,\infty) \neq \varnothing$.  In what follows, without loss of generality, I'll focus almost exclusively on upper probabilities.  De Finetti-style betting interpretations can be given to $\lprior$ and $\uprior$ (see  Remark~\ref{re:betting} in Appendix~\ref{A:remarks}), but this isn't necessary.  You can follow \citet{shafer1976} and others by interpreting $\lprior(H)$ and $\uprior(H)$ simply as Your quantitative degree of confidence/support and degree of plausibility, respectively, in the truthfulness of $H$. 


In precise probability theory, probabilities are extended to expected values via Lebesgue integration.  In imprecise probability theory, there's an analogous extension to lower and upper expected values, or lower and upper previsions \citep[e.g.,][]{walley1991, lower.previsions.book}.  Given $g: \TT \to \RR$, the lower and upper expected value is defined as 
\begin{equation}
\label{eq:envelope}
\lprior \, g = \inf_{\prior \in \credal} \E^{\Theta \sim \prior} \{g(\Theta)\} \quad \text{and} \quad \uprior \, g = \sup_{\prior \in \credal} \E^{\Theta \sim \prior} \{ g(\Theta) \}, 
\end{equation}
i.e., the lower and upper limits of the usual expected values of $g(\Theta)$ with respect to probability distributions $\prior$ in the credal set $\credal$.  
According to the definition above, computation of the lower and upper expectations appears to require non-trivial optimization over a potentially infinite-dimensional space $\credal$.  In certain cases, however, there are equivalent formulations with computationally simpler formulas; see, e.g., Equation~\eqref{eq:upper.e.poss} below and the general discussion of {\em Choquet integration} in Appendix~\ref{A:choquet}.   

A special case of imprecise probability is {\em possibility theory}, 
closely tied to fuzzy set theory \citep[e.g.,][]{zadeh1978}.  The seminal text on possibility theory is \citet{dubois.prade.book}; see, also
\citet{hose2022thesis} for a review and, e.g., \citet{dubois2006} and \citet{imchar, martin.partial2} for connections to statistics.  What distinguishes possibility theory from other brands of imprecise probability is its simplicity.  Indeed, a possibility measure is determined by a real-valued function, analogous to how precise probabilities are determined by a probability density.  The difference is that, while probability theory uses integration, possibility theory is based on optimization.  Define a contour function $q: \TT \to [0,1]$ such that $\sup_{\theta \in \TT} q(\theta) = 1$, the condition analogous to a probability density function integrating to 1.  Then the corresponding possibility measure is the upper probability $\uprior$ defined via optimization:
\begin{equation}
\label{eq:upper.p.poss}
\uprior(H) = \sup_{\theta \in H} q(\theta), \quad H \subseteq \TT.
\end{equation}
A simple illustration of a possibility contour $q$ and of the formula \eqref{eq:upper.p.poss} is shown in Figure~\ref{fig:toy}.  And the ``You're 95\% sure that $\Theta$ exceeds 7'' case can easily be encoded by a possibility measure $\uprior$ as in \eqref{eq:upper.p.poss} determined by the contour $q(\theta) = 0.05 \times 1(\theta \leq 7) + 1(\theta > 7)$.

\begin{figure}[t]
\begin{center}
\scalebox{0.6}{\includegraphics{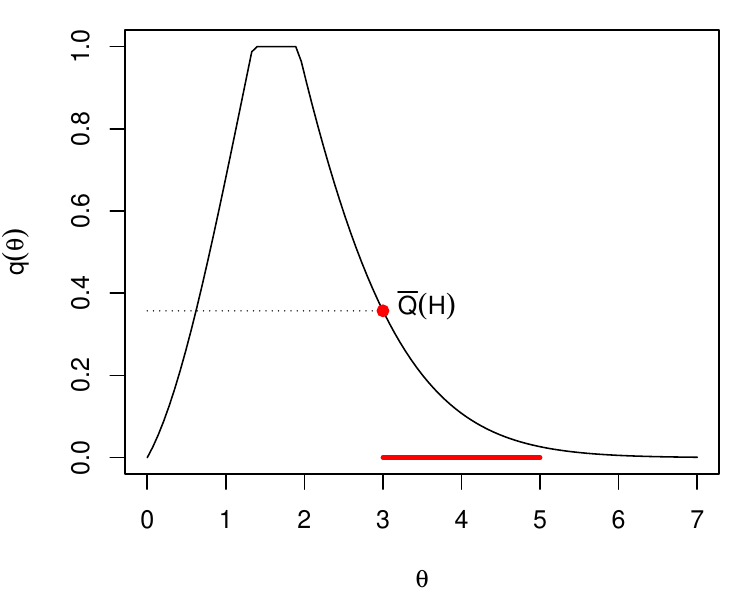}}
\end{center}
\caption{A possibility contour $q$, the hypothesis of interest $H=[3,5]$ (red), and the corresponding upper probability $\uprior(H)$ determined by optimization as in \eqref{eq:upper.p.poss}.}
\label{fig:toy}
\end{figure}

Two possibility-theoretic details deserve mention.  The first is a simple formula for $\uprior$'s extension to an upper expected value.  
If $g: \TT \to \RR$ is a suitable non-negative function and $\uprior$ a possibility measure determined by contour $q$ as described above, then the corresponding possibilistic upper expected value (see Appendix~\ref{A:choquet})
is given by
\begin{equation}
\label{eq:upper.e.poss}
\uprior \, g = \int_0^1 \Bigl\{ \sup_{\theta: q(\theta) > s} g(\theta) \Bigr\} \, ds.
\end{equation}
The second is a simple characterization \citep{cuoso.etal.2001, dubois.etal.2004} of the credal set $\credal$ associated with the possibility measure $\uprior$ determined by contour $q$:
\begin{equation}
\label{eq:credal.char}
\prior \in \credal \iff \prior\{ q(\Theta) \leq \alpha \} \leq \alpha \;\; \text{for all $\alpha \in [0,1]$}. 
\end{equation}
The reader may notice the similarity to p-values: $\prior$ belongs to the credal set if and only if $q(\Theta)$ is stochastically no smaller than $\unif(0,1)$ as a function of $\Theta \sim \prior$.  



\section{Regularized e-processes}
\label{S:reg.eprocess}

\subsection{Prior information and regularization}

Suppose Your knowledge about $\Theta$ takes the form of a (possibly vacuous) imprecise probability $\uprior$.  Introduce $\rho: \TT \to [0,\infty]$ such that $\rho(\theta)$ represents the weight of evidence (relative to $\uprior$) against the basic proposition ``$\Theta=\theta$.''  This is different from familiar notions of statistical evidence because there's no data, but there are still connections.  \citet[][Ch.~6.7]{good.evidence.book}, for example, proves a now-well-known result (which he attributes to Turing) that, in certain cases, the expectation of the Bayes factor weighing evidence against a true proposition is equal to 1.  More generally, this ``expected value is $\leq 1$'' property is key to all the recent e-value developments; see \eqref{eq:eval.bound} above.  To be a proper {\em regularizer}, the function $\rho$ must be similarly bounded in (upper) expectation.

\begin{defn}
\label{def:regularizer}
Let $\uprior$ be an upper probability/prevision on $\TT$.  A function $\rho: \TT \to [0,\infty]$ is a regularizer if its $\uprior$-upper expected value is bounded by 1, i.e., if $\uprior \rho \leq 1$. 
\end{defn}


Note that any non-negative function $\rho$ with $0 \leq \rho \leq 1$ is a regularizer, but such a choice is useless.  Suppose, for example, that the prior is fully vacuous, i.e., if You literally know nothing about $\Theta$ {\em a priori}, which implies $\uprior\rho = \sup_{\theta \in \TT} \rho(\theta)$.  Then choosing $\rho \leq 1$ is the only choice that satisfies the condition in Definition~\ref{def:regularizer}.  So, a {\em trivial} choice of $\rho$ that's upper-bounded by 1 matches the no-prior-information case and, hence, can't offer any meaningful regularization.  Henceforth, except for the vacuous prior case where $\rho \equiv 1$ is the clear choice of regularizer, I'll focus on {\em non-trivial} regularizers, i.e., those that aren't upper-bounded by 1. 
Then the general goal is, for Your given $\uprior$, to find a function $\rho$ which can take values possibly much larger than 1 while still satisfying $\uprior\rho \leq 1$.  Naturally, for two regularizers $\rho$ and $\rho'$ with $\rho'(\theta) \leq \rho(\theta)$ for all $\theta$, the dominant regularizer $\rho$ is preferred.  A concrete transformation of prior information into a regularizer will be presented below; further details and examples are given in Appendix~\ref{A:other}. 

The quality of the regularizer is determined by how much of a knowledge-based boost it offers to an existing e-process $\eval_\theta(\cdot)$ for $\Theta$, as described in Section~\ref{SS:eprocess}.
My proposal is to define a {\em regularized e-process} by combining $\rho$ and $(\eval_\theta: \theta \in \TT)$ as follows: 
\begin{equation}
\label{eq:reg.eprocess}
\eval^\text{reg}(\cdot, \theta) = \rho(\theta) \times \eval_\theta(\cdot), \quad \theta \in \TT.
\end{equation}
The intuition is that, if $\theta$ is incompatible with prior knowledge, so that $\rho(\theta) \gg 1$, then $\eval^\text{reg}(\cdot,\theta) \gg \eval_\theta(\cdot)$, which creates an opportunity for increased efficiency.  But non-trivially manipulating the original e-process in an effort to boost efficiency will surely jeopardize anytime validity in the sense of Section~\ref{SS:eprocess}, hence care is needed.  In Section~\ref{SS:reg.ville} below, I'll establish my main claim: in a certain sense, the regularized e-process enjoys efficiency gains without jeopardizing anytime validity.  This ``certain sense'' involves what I'll argue is a meaningful imprecise-probabilistic relaxation of anytime validity as in Section~\ref{SS:eprocess}.

Finally, there are justifications for my choice to define the regularized e-process in \eqref{eq:reg.eprocess} via multiplication.  One is based on a Bayesian-like updating-coherence property and the other is based on a formal demonstration of dominance, similar to a result in \citet{vovk.wang.2021}; these details are presented in Appendix~\ref{SS:product}.

\subsection{Special case: possibilistic priors}
\label{SS:possibilistic.case}

Suppose Your prior knowledge about $\Theta$ is encoded by a possibility measure with contour function $q$; see Appendix~B for other options.  Then Your $\uprior$ is defined via optimization as in \eqref{eq:upper.p.poss}, and upper expectation is as in \eqref{eq:upper.e.poss}.  This is the simplest imprecise probability model, and simplicity is a virtue: Your prior knowledge about $\Theta$ is necessarily limited, so a more expressive imprecise probability model used to describe it makes elicitation more difficult.  The claim is that experts can offer statements like: ``I'd not be surprised at all if $\Theta=a$, I'd be a little surprised if $\Theta=b$, and I'd be very surprised if $\Theta=c$.'' ({\em Surprise} is due to \citealt{shackle1961}.)  Then the qualitative statement above could be made quantitative by introducing a possibility contour $q$ with $q(a)=1$, $q(b)$ smaller, $q(c)$ very small, etc.  This is often how penalties and priors are chosen.  For example, a Bayesian might take the least-surprising value to be the prior mode and then choose a density that's a decreasing function of surprise.  But there's a drastic difference between a probability density and a possibility contour, even if they have similar shapes: the former determines a precise probability thereby adding artificial information to Your corpus of knowledge.  

Recall that a useful regularizer $\rho$ must take values (potentially much) larger than 1 at $\theta$ values that are incompatible with the prior information.  In the present context, ``incompatibility'' corresponds to a small $q$ value, so $\rho(\theta)$ should be large when $q(\theta)$ is small.  As a first attempt, this could be accomplished by taking $\rho$ to be the reciprocal of $q$.  But $q$ is no more than 1, so such a $\rho$ would never be less than 1 and hence it can't be a regularizer in the sense of Definition~\ref{def:regularizer}.  Apparently $q$ is too small for the reciprocal to work, so it needs to be inflated first.  Following \citet[][Sec.~6]{shafer.vovk.martingale}, define a(n admissible) {\em calibrator}
as a non-decreasing function $\gamma: [0,1] \to (0, \infty]$ such that   
\begin{equation}
\label{eq:calibrate}
\int_0^1 \frac{1}{\gamma(s)} \, ds = 1. 
\end{equation} 
Of course, $\gamma$ can taken to be the reciprocal of any probability density function supported on $[0,1]$.  Then 
I'll define the {\em regularizer} in terms of the calibrated version of $q$: 
\begin{equation}
\label{eq:regularizer}
\rho(\theta) = \{ \gamma \circ q(\theta) \}^{-1}, \quad \theta \in \TT. 
\end{equation}
The role played by the calibrator is to inflate the contour ``just enough.'' 

\begin{prop}
\label{prop:prior.reg}
Let $\uprior$ be the possibility measure determined by the contour function $q$.  Then $\rho$ defined in \eqref{eq:regularizer} is a regularizer, i.e., has $\uprior$-upper expectation bounded by 1.
\end{prop}


I'll follow \citet[][Appendix~B]{vovk.wang.2021} and suggest use of (the reciprocal of) a suitable beta mixture of beta density functions, which takes the form 
\[ \gamma(u) = \frac{u \, (-\log u)^{1 + \kappa}}{\kappa \times {\tt igamma}(-\log u, 1 + \kappa)}, \quad u \in [0,1], \quad \kappa > 0, \]
where ${\tt igamma}(z, \alpha) = \int_0^z t^{\alpha-1} \, e^{-t} \, dt$ is the incomplete gamma function. A particular advantage of the above calibrator is how rapidly it vanishes as $u \to 0$, which is directly related to how severely the corresponding regularizer penalizes those $\theta$ values incompatible with the prior information.  The gamma defined in the above display is the calibrator that I'll use for my illustrations in Section~\ref{SS:gains} and elsewhere.  

Although the context and form here are different, the message above is a familiar one to those who have experience with e-processes, etc.  Indeed, Equation~\eqref{eq:credal.char} implies that, roughly, $q(\Theta)$ is a p-value relative to $\uprior$.  And it's well-known \citep[e.g.,][]{sellkebayarriberger2001, vovk1993} that one must calibrate a p-value so that its reciprocal is an e-value.

\subsection{Induced upper joint distributions}
\label{SS:pullback}

Recall that $\Theta = f(\Omega) \in \TT$ is a function of $\Omega \in \OO$.  Since $\TT$ is defined as the image of $\OO$, the map $f: \OO \to \TT$ is trivially surjective.  There are applications in which $f$ would be a bijection, e.g., when $f$ is the identity mapping, so that $\Theta$ and $\Omega$ are equivalent in some sense.  
But it's typical that $\Theta = f(\Omega)$ is a lower-dimensional feature of $\Omega$.  Whether $f$ is or isn't a bijection becomes relevant when considering how knowledge about $\Theta$ translates to knowledge about the primitive $\Omega$ from which $\Theta=f(\Omega)$ is derived.  The {\em push-forward} operation that takes a probability for $\Omega$ to a corresponding probability for $\Theta=f(\Omega)$ is well-defined, whereas the reverse {\em pull-back} operation is not unique when $f$ is not a bijection.  Here I give only the necessary details concerning pull-backs. 



Given a probability distribution $\prior$ for $\Theta$ on $\TT$, the relevant question is under what conditions does there exist a corresponding probability distribution, say, $\pb$, for $\Omega$ on $\OO$ such that the distribution of $f(\Omega)$ is $\prior$.  The classical result of \citet[][Lemma~2.2]{varadarajan1963} establishes the existence of an $\pb$ corresponding to $\prior$ under very mild conditions.
In fact, for a given $\prior$, there's a non-empty class $\pbcred_\prior$ of such distributions, i.e., 
\begin{equation}
\label{eq:R.Q}
\pbcred_\prior = \{ \pb: \text{if $\Omega \sim \pb$, then $f(\Omega) \sim \prior$} \}, 
\end{equation}
and, moreover, this class is easily shown to satisfy the properties of a credal set.  

\begin{prop}
\label{prop:pullback}
For any given $\prior$ and continuous $f: \OO \to \TT$, the collection $\pbcred_\prior$ in \eqref{eq:R.Q} is non-empty, convex, and closed with respect to the weak topology. 
\end{prop}


An advantage to this credal set characterization is that it allows for the construction of a coherent upper joint distribution for $(Z^N, \Omega)$ based solely on the model $\model$ and Your prior information in $\credal$ about $\Theta$.  
Specifically, let $\uuprob$ denote this induced upper joint distribution for $(Z^N, \Omega)$, which is defined as via its upper expectation as 
\begin{equation}
\label{eq:uuprob}
\uuprob \, g = \sup_{\prior \in \credal} \sup_{\pb \in \pbcred_\prior} \underbrace{\E^{\Omega \sim \pb} \bigl[ \E^{Z \sim \prob_\Omega}\{ g(Z^N, \Omega) \} \bigr]}_{\text{joint distribution expectation}}, 
\end{equation}
where $g: \ZZ^\infty \times \OO \to \RR$ is a suitable function, and the highlighted term is just the usual expectation of $g(Z^N, \Omega)$ with respect to the joint distribution determined by the (marginal) distribution $\pb$ for $\Omega$ and the (now-interpreted-as-a-conditional) distribution $\prob_\omega$ for $Z^N$, given $\Omega=\omega$.  As mentioned in Section~\ref{SS:ip}, $\uuprob$-upper probabilities correspond to the right-hand side above with the appropriate indicator function plugged in for $g$.  

\subsection{Regularized Ville's inequality}
\label{SS:reg.ville}

My claim is that the regularized e-process satisfies prior knowledge-dependent versions of property \eqref{eq:eval.bound} and of Ville's inequality \eqref{eq:ville}.  This leads to a corresponding ``regularized'' version of anytime validity discussed further in Section~\ref{SS:implications} below. 

\begin{thm}
\label{thm:reg.ville}
Suppose the available prior information is encoded in the upper probability $\uprior$, which determines an upper joint distribution $\uuprob$ as in \eqref{eq:uuprob}.  If $\rho$ is a regularizer in the sense of Definition~\ref{def:regularizer} relative to $\uprior$, then for any e-process $\{\eval_\theta(\cdot): \theta \in \TT\}$, the corresponding regularized version $\eval^\text{\rm reg}$ in \eqref{eq:reg.eprocess} satisfies 
\begin{equation}
\label{eq:reg.ville0}
\uuprob\bigl[ \eval^\text{\rm reg}\{Z^N,f(\Omega)\} \bigr] \leq 1, \quad \text{for all stopping times $N$}, 
\end{equation}
and the following regularized Ville's inequality holds:
\begin{equation}
\label{eq:reg.ville}
\uuprob \bigl[ \eval^\text{\rm reg}\{Z^N, f(\Omega)\} > \alpha^{-1} \bigr] \leq \alpha \quad \text{all $\alpha \in [0,1]$, all stopping times $N$}. 
\end{equation}
\end{thm}

The above result is related to the``uniformly-randomized Markov inequality'' in \citet[][Theorem~1.2]{ramdas.randomized.markov}.  They prove that, if $X$ and $U$ are independent random variables, with $X \geq 0$ and $U$ stochastically no smaller than $\unif(0,1)$, then $\prob(X / U \geq a^{-1}) \leq a \,\E(X)$ for $a > 0$. 
The connection between their randomized Markov inequality and the result in Theorem~\ref{thm:reg.ville} is most clear in the special case where $\rho$ is defined in terms of the possibility contour $q$ as in Section~\ref{SS:possibilistic.case}.  In such a case, the random variable $q(\Theta)$ is no smaller than $\unif(0,1)$ when $\Theta \sim \prior$, for any prior $\prior$ in the credal set $\credal$.  Since the calibration step $\gamma \circ q(\Theta)$ boosts it further, the proposed regularization is simply dividing the usual e-process by a no-smaller-than-uniform random variable.  

\subsection{Statistical implications}
\label{SS:implications}

Theorem~\ref{thm:reg.ville}'s most important take-away is its implication concerning safe, anytime valid inference.  Indeed, when relevant, non-vacuous prior information is available,  Theorem~\ref{thm:reg.ville} shows how that knowledge can be used to enhance an e-process in such a way that anytime validity is preserved but efficiency is generally gained.  This enhancement is achieved through the incorporation of a regularizer as in \eqref{eq:reg.eprocess} that inflates and deflates the original e-process when the latter is large and small, respectively.  If the goal is to test ``$\Theta \in H$,'' then the regularized e-process-based testing procedure  
\[ \text{reject ``$\Theta \in H$'' based on data $z^n$} \iff \inf_{\theta \in H} \eval^\text{reg}(z^n,\theta) > \alpha^{-1} \]
will be anytime valid (relative to Your prior knowledge) in the sense that 
\begin{align*}
\uuprob(\text{Type~I error}) := & \; \uuprob\{ \text{$f(\Omega) \in H$ and test of `$f(\Omega) \in H$' rejects} \} \\
= & \; \uuprob\Bigl\{ f(\Omega) \in H \text{ and } \inf_{\theta \in H} \eval^\text{reg}(Z^N,\theta) > \alpha^{-1} \Bigr\} \\
\leq & \; \uuprob\bigl\{ \eval^\text{reg}(Z^N,f(\Omega)) > \alpha^{-1} \bigr\} \\
\leq & \; \alpha,
\end{align*}
where the last line follows by \eqref{eq:reg.ville}.  This is a non-trivial generalization of the familiar frequentist Type~I error control so it warrants some remarks.  In the statistics literature, the prevailing viewpoint is that the quantity of interest $\Theta$ is a fixed unknown, so the statement ``$\Theta \in H$'' is absolutely true for some $H$ and absolutely false for others, but nothing more can be said.  This aligns with the vacuous-prior case where, with the exception of the trivial hypotheses $\varnothing$ and $\TT$, Your corpus of knowledge makes exactly the same statements about every hypothesis: $\lprior(H) = 0$ and $\uprior(H)=1$.  That is, there's no evidence supporting the truthfulness of either ``$\Theta \in H$'' or ``$\Theta \not\in H$.'' With this extreme form of prior knowledge, Your corresponding regularizer is $\rho \equiv 1$---so that $\eval^\text{reg}(\cdot, \theta) = \eval_\theta(\cdot)$---and $\uuprob$ simplifies (see Remark~\ref{re:vac.supremum} in Appendix~\ref{A:remarks}) to yield the following:
\begin{align}
\uuprob(\text{Type~I error}) := & \; \uuprob\{ \text{$f(\Omega) \in H$ and test of `$f(\Omega) \in H$' rejects}\} \notag \\
= & \; \sup_{\omega: f(\omega) \in H} \prob_\omega \Bigl\{ \inf_{\theta \in H} \eval_\theta(Z^N) > \alpha^{-1} \Bigr\} \label{eq:vac.supremum}. 
\end{align}
Then the anytime Type~I error control property of the original e-process-based test, derived from \eqref{eq:ville}, is recovered as a special case of Theorem~\ref{thm:reg.ville}.  
Beyond this extreme case, the different $H$'s have varying degrees of belief/plausibility and this is taken into account in the evaluation of the test via the conjunction: ``$f(\Omega) \in H$'' {\em and} ``test of `$f(\Omega) \in H$' rejects.'' In particular, monotonicity of $\uuprob$ implies that $\uuprob(\text{Type~I error}) \leq \uprior(H)$, which means that the test doesn't have to be particularly good for those $H$ with small $\uprior(H)$ to control the (generalized) Type~I error rate.  If (generalized) Type~I error control is easy for those {\em a priori} implausible $H$'s, then that gives the e-process an opportunity to focus its effort on handling the plausible $H$'s more efficiently.  

Similarly, a nominal $100(1-\alpha)$\% confidence region for $\Theta$, based on the regularized e-process applied to data $z^n$, is 
\[ C_\alpha^\text{reg}(z^n) = \bigl\{ \theta \in \TT: \eval^\text{reg}(z^n, \theta) \leq \alpha^{-1} \bigr\}, \]
i.e., the collection of simple null hypotheses that the regularized e-process-based test would not reject at level $\alpha$.  Then the non-coverage (upper) probability is 
\begin{equation}
\label{eq:reg.noncvg}
\uuprob\{ C_\alpha^\text{reg}(Z^N) \not\ni f(\Omega) \} = \uuprob\{ \eval^\text{reg}(Z^N, f(\Omega)) > \alpha^{-1} \} \leq \alpha, 
\end{equation}
which implies that the corresponding lower probability of coverage, i.e., of $C_\alpha^\text{reg}(Z^N) \ni f(\Omega)$, is bounded from below by $1-\alpha$.  This justifies calling $C_\alpha^\text{reg}$ a (generalized) ``anytime $100(1-\alpha)$\% confidence region.'' Like in the discussion above, this is generally different from the usual notion of coverage probability of a confidence set.  In the case of vacuous prior information, however, the regularizer is $\rho \equiv 1$, the corresponding confidence region is $C_\alpha(z^n) = \{ \theta \in \TT: \eval_\theta(z^n) \leq \alpha^{-1} \}$, and inequality  \eqref{eq:reg.noncvg} reduces to 
\[ \sup_\omega \prob_\omega\{ C_\alpha^\text{reg}(Z^N) \not\ni f(\omega) \} \leq \alpha. \]
This, of course, is exactly the familiar anytime coverage probability guarantees offered by e-process-based confidence sets, via \eqref{eq:ville}.  


To summarize, my proposal involves two key steps: first, a regularizer is created and fused with a given e-process in such a way that the resulting regularized e-process tends to be larger at points incompatible with prior knowledge and hence more efficient than the given e-process alone; second, the definition of ``anytime validity'' is correspondingly relaxed, via the prior knowledge-dependent $\uuprob$, to accommodate the broadly larger and more efficient regularized e-process.  Of course, whether a property like ``the method's answer is wrong with small $\uuprob$-probability'' carries weight with readers, peer reviewers, collaborators, etc.~hinges on whether Your prior knowledge incorporated in $\uuprob$ is justified.  The same can be said for any kind of model-based analysis: the conclusions drawn can be no more convincing than the justification given for the assumed model.  But compared to a Bayesian framework that requires prior knowledge be encoded as a precise probability distribution, here I'm placing no serious constraints on the form Your prior knowledge takes; You're free to use whatever $\uprior$ You can justify, even if it's vacuous.  Your target audience is scientifically obligated to scrutinize Your justification, but if Your beliefs warrant being called {\em prior knowledge}, You must have strong justification for them and so passing this scrutiny should only be a matter of explaining Your reasoning clearly.

\subsection{Efficiency gains}
\label{SS:gains}

The goal here is to demonstrate the efficiency that can be gained through the incorporation of partial prior information via my proposed regularized e-process framework.  I'll take a simple model and e-process, so that effort can be focused on the possibilistic prior, the corresponding regularizer, and its effect on efficiency.  I must emphasize that {\em I'm not recommending off-the-shelf use of any particular regularizer}.  It's Your responsibility to determine what, if any, relevant prior information about $\Theta$ is available and how to quantify it.  I'm in no position to say what You should or shouldn't believe about $\Theta$.  

To set the scene, consider a normal mean model with known variance equal to 1.  I'm assuming a parametric model, hence the uncertain model index, $\Omega$, and the quantity of interest, $\Theta$, are the same.  So I'll write $\model = \{ \prob_\theta: \theta \in \TT\}$, where $\prob_\theta = \nm(\theta, 1)$ and $\TT = \RR$.  For the e-process, I'll consider the Bayes factor
\[ \eval_\theta(\cdot) = \frac{\int L_t(\cdot) \, \xi(t) \, dt}{L_\theta(\cdot)}, \]
where $L_\theta$ is the usual Gaussian likelihood function and, in the numerator, $\xi$ is a mixing probability density function on $\TT$.  In what follows, I'll take $\xi$ to be the $\nm(0, v)$ density function, with $v=10$.  This is a version of the so-called Savage--Dickey e-process \citep[e.g.,][]{grunwald.epost}.  The integration can be done in closed-form, so
\[ \eval_\theta(z^n) = (nv + 1)^{1/2} \exp\bigl\{ -\tfrac{n}{2} (\theta - \bar z_n)^2 + \tfrac12 ( \tfrac{n}{nv + 1} ) \, \bar z_n^2 \bigr\}, \quad \theta \in \TT. \]
In words, the e-process is a ratio of the (Bayesian marginal) likelihood under the model where $\Theta$ is different from $\theta$---where ``different from $\theta$'' is quantified by $\xi$---to the likelihood under the model where $\Theta$ equals $\theta$.  It's important to point out that $\xi$ isn't genuine prior information, it's just a (default) choice that's made to define the e-process.  
The mixing distributions may be relatively diffuse, hence my choice of $\xi$'s variance as $v=10$.  

Here I'll describe the construction of a partial prior that takes the form of a possibility measure, based on an assessment of {\em surprise}, as mentioned briefly above.  This is not the only kind of partial prior one can consider, and some additional details and examples are shown in Appendix~\ref{S:emp.efficiency}.  The strategy here, which is quite common, is to assign Your surprise by analogy, by making comparison to familiar probability assignments.  For example, suppose that, for each $t > 0$, You'd be as surprised to learn that $|\Theta| > t$ as you would be to observe $|Y| > t$, where $Y \sim \nm(0,K)$ for a specified $K$.  Matching Your surprise to the latter probability amounts to choosing a prior contour 
\[ q(\theta) 
= 1 - {\tt pchisq}(\theta^2 / K, {\tt df}=1), \quad \theta \in \TT. \]
This is a principled possibility measure construction, and further details can be found in Remark~\ref{re:prob.to.poss} of Appendix~\ref{A:remarks}. 
Importantly, what I just described is entirely different from You taking $\nm(0,K)$ as a prior distribution for $\Theta$, since $\nm(0,K)$ is just one of the many probability measures contained in the possibilistic prior's credal set. 


Figure~\ref{fig:gauss} plots the regularized (and unregularized) log-transformed e-process as a function of $\theta$ for the ``Gaussian prior'' described above, for $K \in \{0.1, 0.2, 0.4, 0.8\}$, and for three different values of the observed sample mean $\bar z$ based on a sample of size $n=5$:
\begin{itemize}
\item $\bar z=0.25$ is consistent with all the priors;
\vspace{-2mm}
\item $\bar z=0.5$ is marginally inconsistent with the priors, and;
\vspace{-2mm}
\item $\bar z=1$ is rather inconsistent with the priors.
\end{itemize}
The black line corresponds to the unregularized e-process, and the four colored lines correspond to the different values of $K$; the dashed horizontal line corresponds to $-\log0.05 \approx 3$, which is the cutoff that determines an e-process's 95\% confidence interval.  As expected, the stronger the prior information (i.e., smaller $K$), the greater the efficiency gains.   

\begin{figure}[t]
\begin{center}
\subfigure[$\bar z=0.25$]{\scalebox{0.6}{\includegraphics{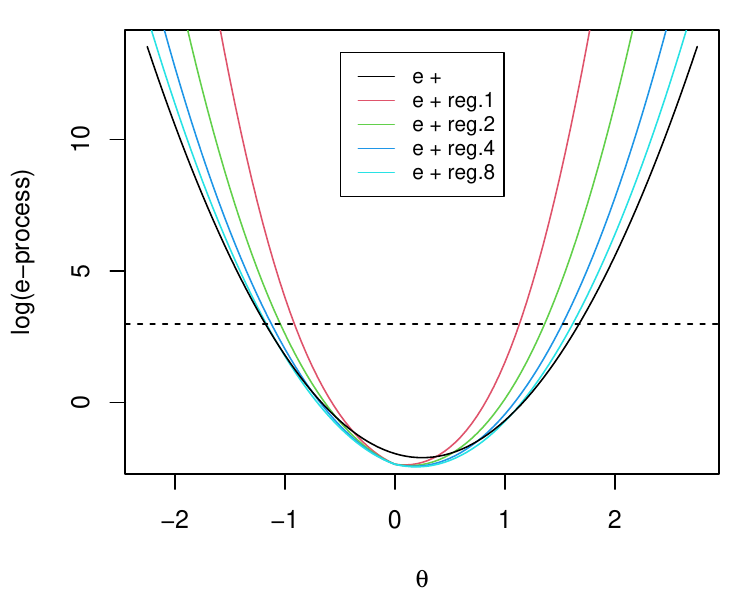}}}
\subfigure[$\bar z=0.5$]{\scalebox{0.6}{\includegraphics{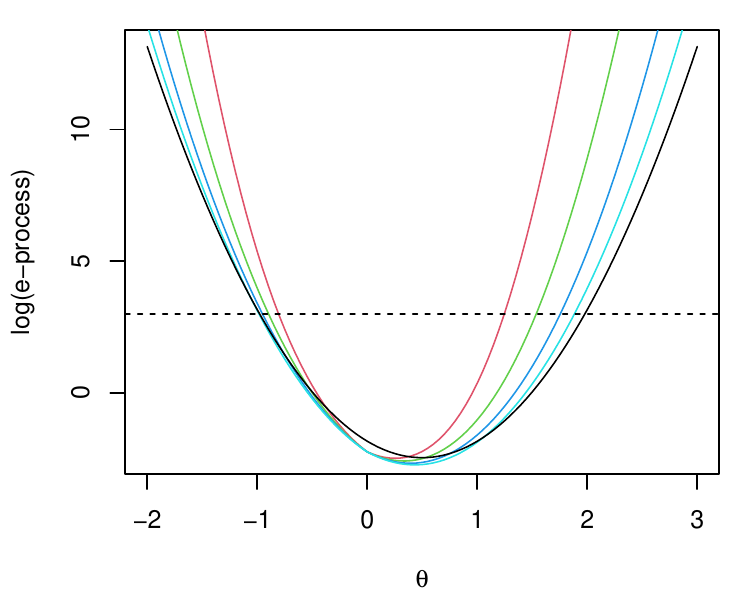}}}
\subfigure[$\bar z=1$]{\scalebox{0.6}{\includegraphics{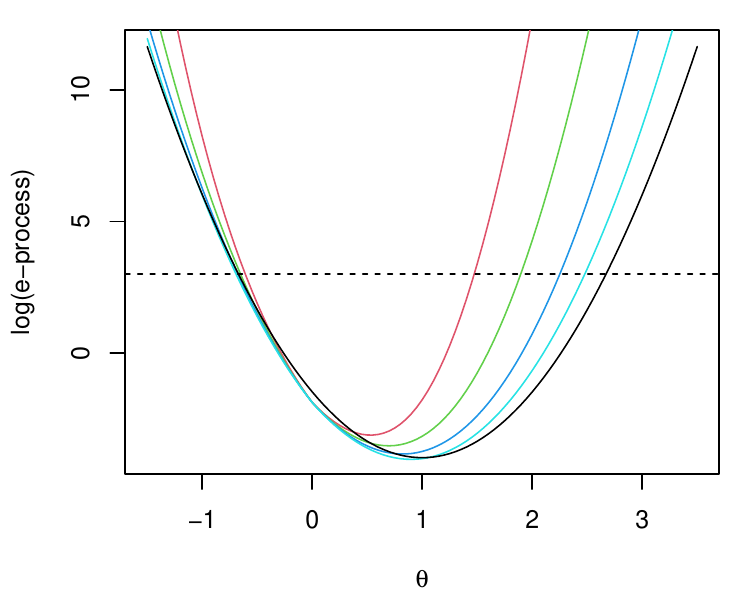}}}
\end{center}
\caption{Plot of $\theta \mapsto \eval^\text{reg}(z^n, \theta)$ for three different data sets $z^n$.}
\label{fig:gauss}
\end{figure}

Another perspective on the efficiency gain concerns the number of sample points required before a false hypothesis can be rejected.  Here I'll consider the hypothesis ``$\Theta=0.7$'' and track the (log) regularized e-process's growth as the sample size increases.  The fewer number of sample points required to reject the false hypothesis, the more efficient the procedure is.  Figure~\ref{fig:gauss.grow} plots the path $n \mapsto \text{avg}\{\log \eval^\text{reg}(z^n, 0.7)\}$, where the average is taken over 1000 replications of the data-generating process.  Clearly, the strongest regularization ($K=0.1$) makes it possible to detect the discrepancy of 0.7 with generally fewer samples, whereas the other priors, which offer weaker regularization, take more samples on average to reach the desired conclusion.  

\begin{figure}[t]
\begin{center}
\scalebox{0.6}{\includegraphics{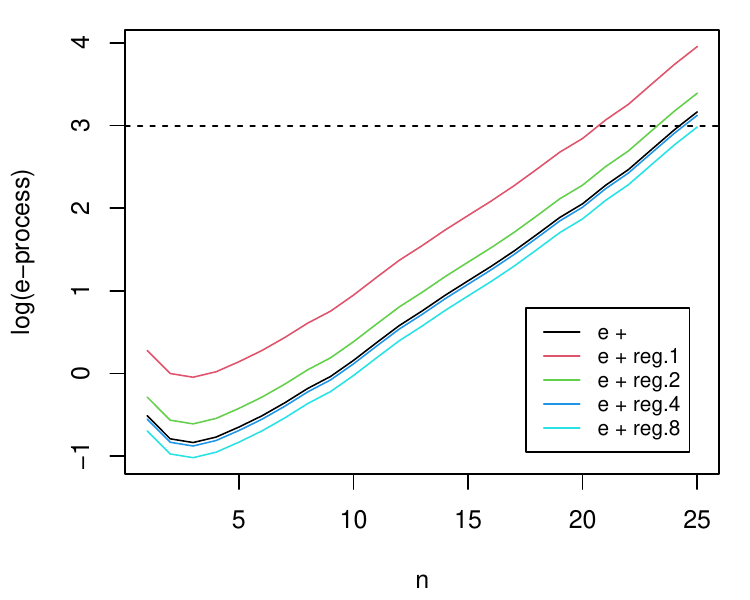}}
\end{center}
\caption{Plot of $n \mapsto \text{avg}\{\log \eval^\text{reg}(z^n, 0.7)\}$ when data are sampled from $\nm(0,1)$.}
\label{fig:gauss.grow}
\end{figure}

\section{Regularized e-uncertainty quantification}
\label{S:reg.uq}

\subsection{Objective}

So far, I've treated inference as the construction of suitable test and confidence set procedures.  In this section, however, I have a more ambitious goal of reliable, data-driven, uncertainty quantification about the uncertain $\Theta$.  

It's a mathematical fact that uncertainty arising in statistical inference generally cannot be quantified reliably using ordinary probability.  Indeed, the {\em false confidence theorem} \citep{balch.martin.ferson.2017} says, roughly, every probabilistic quantification of uncertainty---Bayes, fiducial, etc.---tends to assign high confidence, i.e., high posterior probability, to certain false hypotheses \citep{martin.nonadditive, imchar, martin.belief2024}.  False confidence creates a risk of unreliability, or ``systematically misleading conclusions'' \citep{reid.cox.2014}.  
{\em Inferential models} (IMs), originally developed in \citet{imbasics, imbook}, offer a new framework for possibilistic statistical reasoning that's reliable in the sense of avoiding false confidence; see the recent review in \citet{imreview}.   Remark~\ref{re:why.possibility} in Appendix~\ref{A:remarks} explains why the possibilistic form is right for statistical problems. 



\subsection{Possibilistic IMs}
\label{SS:ims}

My current efforts \citep{martin.partial2, martin.partial3, martin.basu} focus on the construction of {\em possibilistic IMs}, where the IM output takes the mathematical form of a possibility measure as reviewed briefly in Section~\ref{SS:ip}.  The specific proposal put forward in the above references starts with defining the IM's possibility contour, based on observed data $Z^n=z^n$, as
\[ \pi_{z^n}(\theta) = \sup_{\omega: f(\omega) = \theta} \prob_\omega\{ r(Z^n,\theta) \leq r(z^n,\theta)\}, \quad \theta \in \TT, \]
where $r(z^n,\theta)$ provides a ranking of the parameter value $\theta$ in terms of its compatibility with $z$---large values indicate higher compatibility.  For example, in \citet{martin.partial2, martin.partial3}, I recommended taking $r$ to be the relative profile likelihood 
\[ R(z^n,\theta) = \frac{\sup_{\omega: f(\omega)=\theta} L_{z^n}(\omega)}{\sup_\omega L_{z^n}(\omega)}, \quad \theta \in \TT, \]
with $\omega \mapsto L_{z^n}(\omega)$ the likelihood function corresponding to the ``iid $\prob_\omega$'' model. Then the possibilistic IM's upper probability output is, as in \eqref{eq:upper.p.poss}, given by optimization: 
\begin{equation}
\label{eq:upi}
\uPi_{z^n}(H) = \sup_{\theta \in H} \pi_{z^n}(\theta), \quad H \subseteq \TT. 
\end{equation} 
The interpretation is that a small $\uPi_{z^n}(H)$ means there's strong evidence in data $z^n$ against the truthfulness of $H$.  That is, a small upper probability assigned to $H$ implies doubt and, therefore, I'd be inclined to ``reject'' $H$.  But how small is ``small''?  Clearly some calibration of the IM's numerical output is needed in order for this line of reasoning to be reliable, i.e., to ensure that I don't systematically doubt true hypotheses or buttress false hypotheses.  The possibilistic IM's calibration property is
\begin{equation}
\label{eq:valid}
\sup_{\omega: f(\omega) \in H} \prob_\omega\{ \uPi_{Z^n}(H) \leq \alpha \} \leq \alpha, \quad \alpha \in [0,1], \quad H \subseteq \TT. 
\end{equation}
In words, \eqref{eq:valid} says that the IM assigning small upper probability values to true hypotheses about $\Theta$ is itself a small probability event.  Since there's an explicit link between the two notions of ``small,'' this offers the calibration needed to avoid systematically erroneous inferences.  
The analogous property concerning the IM's lower probabilities---recall the conjugacy relationship, $\lPi_{z^n}(H) = 1 - \uPi_{z^n}(H^c)$, from Section~\ref{SS:ip}---reads as follows: 
\[ \sup_{\omega: f(\omega) \not\in H} \prob_\omega\{ \lPi_{Z^n}(H) \geq 1-\alpha \} \leq \alpha, \quad \alpha \in [0,1], \quad H \subseteq \TT. \]
This latter result says that the IM assigning large lower probability---or ``confidence''---to false hypotheses is a small-probability event; now it should be clear why I say that the IM is safe from the risk of false confidence.  

It's based on the property \eqref{eq:valid} that I say the IM is {\em valid} and, in turn, that uncertainty quantification based on the IM output is reliable.  Among other things, it's straightforward to establish that test and confidence procedures derived from the IM output achieve the usual frequentist error rate control guarantees.  


\subsection{Connection to e-processes}
\label{SS:im.e}

The discussion above focused on the case of a fixed sample size.  It's well-known, however, that the fixed-$n$ sampling distribution properties---such as validity in \eqref{eq:valid}---generally fail when the stopping rule $N$ is data-dependent.  
If the stopping rule $N$ that's used is {\em known}, then this can be easily incorporated into the IM construction such that validity is achieved.  
The trouble, of course, is that the $N$ employed in the data-collection process is often {\em unknown}.
\citet{martin.basu} suggested a conceptually simple work-around that interprets $N$ as a ``nuisance parameter.'' Then the general rules in \citet{martin.partial3} for handling nuisance parameters immediately suggest a new IM contour 
\[ \pi_{z^n}(\theta) = \sup_N \sup_{\omega: f(\omega)=\theta} \prob_\omega\{ r(Z^N,\theta) \leq r(z^n,\theta)\}, \quad \theta \in \TT, \]
where the outermost supremum is over all the stopping rules in consideration (that are consistent with the observation ``$N=n$'').  Of course, direct computation of the right-hand side above can be challenging, depending on the complexity of the set of stopping rules in consideration.  It's here that e-processes come in handy.  

Recall that the ranking function $r$ in the possibilistic IM construction is quite flexible.  One alternative to the proposal above is to take $r$ as the reciprocal of an e-process: $r(z^n, \theta) = \eval_\theta(z^n)^{-1}$. In that case, the usual Ville's inequality gives 
\begin{align*}
\pi_{z^n}(\theta) & = \sup_N \sup_{\omega: f(\omega)=\theta} \prob_\omega\{ r(Z^N,\theta) \leq r(z^n,\theta)\} \\
& = \sup_N \sup_{\omega: f(\omega)=\theta} \prob_\omega\{ \eval_\theta(Z^N) \geq \eval_\theta(z^n) \} \\
& \leq 1 \wedge \eval_\theta(z^n)^{-1}.
\end{align*}
Set $\pi_{z^n}^\eval(\theta) = 1 \wedge \eval_\theta(z^n)^{-1}$ to be the upper bound, so that 
\[ \pi_{z^n}(\theta) \leq \pi_{z^n}^\eval(\theta) := 1 \wedge \eval_\theta(z^n)^{-1}, \quad \theta \in \TT. \]
This upper bound corresponds to what \citet{grunwald.epost} calls his capped {\em e-posterior}.  But there are two further observations about this bound that are worth noting.  First, the bound itself is typically (see Remark~\ref{re:typically} in Appendix~\ref{A:remarks}) a possibility contour.  So, there's a corresponding possibilistic IM, with coherent upper probability $\uPi_{z^n}^\eval$ defined via optimization as in \eqref{eq:upi}.  This determines an {\em e-possibilistic IM} for uncertainty quantification about $\Theta$.  Second, that $\pi_{z^n}^\eval$ is an upper bound of the contour $\pi_{z^n}$, and that the latter determines an anytime valid IM, immediately implies that the IM corresponding to the former is anytime valid too.  This discussion is summarized as

\begin{thm}
\label{thm:valid.vac}
Given an e-process $\eval$, the corresponding e-possibilistic IM with upper probability determined by optimization of the contour $\pi_{z^n}^\eval$, i.e.,  
\[ \uPi_{z^n}^\eval(H) = \sup_{\theta \in H} \pi_{z^n}^\eval(\theta), \quad H \subseteq \TT, \]
is anytime valid in the sense that 
\begin{equation}
\label{eq:valid.vac}
\sup_{\omega: f(\omega) \in H} \prob_\omega\{ \uPi_{Z^N}^\eval(H) \leq \alpha \} \leq \alpha, \quad \text{all $\alpha \in [0,1]$, all $N$, all $H \subseteq \TT$}. 
\end{equation}
\end{thm}

%

As in Section~\ref{SS:ims}, this calibration property is important because, without it, there'd be no meaningful justification for any particular interpretation of the numerical values an e-possibilistic IM assigns to various inputs $(z^n, H)$.  But in light of Theorem~\ref{thm:valid.vac}, the same reliability-guaranteeing ``no-false-confidence'' statement made above for the simple possibilistic IM also holds here for the new e-possibilistic IM.  

One important consequence of anytime validity is that one can derive test procedures from the IM and these will inherit the desired frequentist properties.  Indeed, for a given hypothesis $H \subset \TT$ about $\Theta$, and a given level $\alpha \in [0,1]$, the test procedure 
\[ \text{reject ``$\Theta \in H$'' if and only if $\uPi_{Z^N}^\eval(H) \leq \alpha$} \]
controls the frequenist Type~I error rate at the specified level $\alpha$.  This result is ``obvious'' because $\uPi_{Z^N}^\eval(H)$ can be small if and only if $\eval_\theta(Z^N)$ is large for all $\theta \in H$: in particular, 
\[ \uPi_{Z^N}^\eval(H) \leq \alpha \iff \inf_{\theta \in H} \eval_\theta(Z^N) \geq \alpha^{-1}. \]
If the hypothesis is true, i.e., if $\Theta = f(\Omega) \in H$, then the right-most event in the above display implies $\eval_{f(\Omega)}(Z^N) \geq \alpha^{-1}$, which, by Ville's inequality \eqref{eq:ville}, has $\prob_\Omega$-probability no more than $\alpha$, uniformly over stopping times $N$.  

For the case of vacuous prior information, there's a consequence of anytime validity \eqref{eq:valid.vac} that deserves mention; this property is used in the proof of Theorem~\ref{thm:valid.vac}.  For the more general prior information cases to be considered below, this new property is actually stronger than \eqref{eq:valid.vac}, but here the two are equivalent.  

\begin{cor}
\label{cor:strong.valid.vac}
Given an e-process $\eval$, the corresponding e-possibilistic IM is strongly anytime valid in the sense that its contour satisfies 
\begin{equation}
\label{eq:strong.valid.vac}
\sup_{\omega \in \OO} \prob_\omega\bigl[ \pi_{Z^N}^\eval\{ f(\omega) \} \leq \alpha \bigr] \leq \alpha, \quad \text{all $\alpha \in [0,1]$, all $N$}.
\end{equation}
\end{cor}


Two comments related to Corollary~\ref{cor:strong.valid.vac} are in order.  First, similar to the points about the construction of anytime valid testing procedures, it's straightforward to construct confidence sets too.  Given a significance level $\alpha \in [0,1]$, a $100(1-\alpha)$\% confidence set derived from the e-process-based IM is the $\alpha$-level set defined by the contour:
\[ C_\alpha^\eval(Z^N) = \{ \theta \in \TT: \pi_{Z^N}^\eval(\theta) > \alpha\}. \]
Then strong anytime validity implies that this is, indeed, a genuine anytime valid confidence set in the sense that $\sup_{\omega \in \OO} \prob_\omega\{ C_\alpha^\eval(Z^N) \not\ni f(\omega) \} \leq \alpha$ for all stopping times $N$.
But, by definition of $\pi^\eval$, it's easy to see that $C_\alpha^\eval$ is identical to
the (unregularized) confidence set determined by the e-process $\eval$.  Then the above non-coverage probability bound follows immediately from the relevant properties of $\eval$ presented above.  

Second, to see why this strong anytime validity \eqref{eq:strong.valid.vac} is ``stronger'' than anytime validity \eqref{eq:valid.vac}, note that monotonicity of possibility measures implies
\begin{equation}
\label{eq:monotone}
\uPi_{Z^N}^\eval(H) \geq \uPi_{Z^N}^\eval(\{f(\omega)\}) \equiv \pi_{Z^N}^\eval(f(\omega)) \quad \text{for all $H$ with $H \ni f(\omega)$}.
\end{equation} 
Interestingly, the above holds even for hypotheses $H$ that depend on data $Z^N$ in some way, e.g., if an adversary with access to the data $Z^N$ is trying to dupe the statistician by selecting a ``most difficult'' hypothesis $H$ post hoc.  This suggests that strong anytime validity is equivalent to a {\em uniform-in-hypotheses} version of anytime validity, and the following corollary states this explicitly; see, also, \citet{cella.martin.probing}.  Note the difference between \eqref{eq:uniform.valid.vac} below and \eqref{eq:valid.vac} above: the former has varying $H$ on the {\em inside} of the probability statement whereas the latter has varying $H$ on the {\em outside}.  

\begin{cor}
\label{cor:uniform.vac}
The e-process-based possibilistic IM constructed above is uniformly anytime valid in the sense that 
\begin{equation}
\label{eq:uniform.valid.vac}
\sup_{\omega \in \OO} \prob_\omega\{ \uPi_{Z^N}^\eval(H) \leq \alpha \text{ for some $H$ with $H \ni f(\omega)$} \} \leq \alpha, \quad \text{all $\alpha \in [0,1]$, all $N$}. 
\end{equation}
\end{cor}


Therefore, the e-possibilistic IM constructed here offers reliable, anytime valid uncertainty quantification that's safe not just against fixed hypotheses as in \eqref{eq:valid.vac}, but also against possibly adversarial or otherwise data-dependent hypotheses as in \eqref{eq:uniform.valid.vac}.  This is important for at least two reasons.  First, investigators feeling ``publish-or-perish'' pressures might succumb to temptations to explore for hypotheses that are incompatible with their data, to secure a statistically significant result.  A commitment to using e-possibilistic IMs prevents these sociological temptations from causing harm. Second, as \citet{mayo.book.2018} articulates, investigators yearn for more than null hypothesis significance tests---they want to probe 
for hypotheses that might be compatible with their data, without risking unreliability.  Multiplicity corrections can't accommodate this.

\subsection{Partial priors and regularization}
\label{SS:euq.properties}

Section~\ref{S:reg.eprocess} showed how prior knowledge can be encoded as a regularizer and then combined with a given e-process such that the resulting regularized e-process is anytime valid in a relaxed sense.  The next step is to flip this regularized e-process into a regularized e-possibilistic IM and to establish the corresponding anytime validity properties in the context of uncertainty quantification.  The remainder of this subsection details these next steps.  Section~\ref{SS:euq.behavior} below deals with some more nuanced behavioral properties.

Define the regularized e-possibilistic IM contour function for $\Theta$, given $z^n$, as 
\begin{equation}
\label{eq:contour.reg}
\pi_{z^n}^{\eval \times \rho}(\theta) = 1 \wedge \eval^\text{reg}(z^n, \theta)^{-1}, \quad \theta \in \TT. 
\end{equation}
The corresponding possibility lower and upper probabilities are denoted as $\lPi_{z^n}^{\eval \times \rho}$ and $\uPi_{z^n}^{\eval \times \rho}$, respectively, with the latter defined via optimization and the former via conjugacy, as usual.  By definition of $\pi^{\eval \times \rho}$ above, the level sets 
\[ C_\alpha^{\eval \times \rho}(z^n) = \{ \theta \in \TT: \pi_{z^n}^{\eval \times \rho}(\theta) > \alpha \}, \quad \alpha \in [0,1], \]
are identical to the regularized e-process-based confidence regions $C_\alpha^\text{reg}$ in Section~\ref{SS:implications}.  Therefore, the sets $C_\alpha^{\eval \times \rho}$ inherit the same non-coverage upper probability bound as in \eqref{eq:reg.noncvg}, hence can be referred to as (generalized) ``anytime valid confidence regions.'' 

Theorem~\ref{thm:strong.valid} below extends Corollary~\ref{cor:strong.valid.vac} to the regularized e-possibilistic IM case.  It's an immediate consequence of Theorem~\ref{thm:reg.ville} and the definition of $\pi^{\eval \times \rho}$.  Recall the upper joint distribution $\uuprob$ for $(Z^N,\Omega)$, defined in \eqref{eq:uuprob} above.  

\begin{thm}
\label{thm:strong.valid}
Given an e-process $\eval$ and regularizer $\rho$, the corresponding regularized e-possibilistic IM with contour \eqref{eq:contour.reg} is strongly anytime valid in the sense that it satisfies 
\begin{equation}
\label{eq:strong.valid}
\uuprob\bigl[ \pi_{Z^N}^{\eval \times \rho}\{ f(\Omega) \} \leq \alpha \bigr] \leq \alpha, \quad \text{all $\alpha \in [0,1]$, all $N$}.
\end{equation}
\end{thm}

This is a generalization of Corollary~\ref{cor:strong.valid.vac} because, if the prior information is vacuous and $\rho \equiv 1$, then $\pi_{Z^N}^{\eval \times \rho} \equiv \pi_{Z^N}^\eval$ and the upper joint distribution $\uuprob$ reduces to the supremum of the $\prob_\omega$-probabilities over all $\omega \in \OO$ as in \eqref{eq:strong.valid.vac}. 


\begin{cor}
\label{cor:strong.valid}
For the same regularized e-possibilistic IM considered in Theorem~\ref{thm:strong.valid}, the following hold for all thresholds $\alpha \in [0,1]$ and all stopping times $N$:
\begin{align*}
\uuprob\bigl\{ \text{there exists $H$ with $H \ni f(\Omega)$ and $\uPi_{Z^N}^{\eval \times \rho}(H) \leq \alpha$} \bigr\} & \leq \alpha \\
\uuprob\bigl\{ \text{$H \ni f(\Omega)$ and $\uPi_{Z^N}^{\eval \times \rho}(H) \leq \alpha$} \bigr\} & \leq \alpha, \quad H \subseteq \TT. 
\end{align*}
\end{cor}

The interpretation is, as before, in terms of reliability.  The second line says that, for any fixed $H$, the upper probability that $H$ is true {\em and} the IM assigns it small upper probability is small.  The first line is stronger, it says that the probability the IM assigns small upper probability to any true hypothesis about $f(\Omega)$ is small.  
To make sense of this, suppose data $z^n$ is observed and, for a relevant hypothesis, it happens that $\uPi_{z^n}^{\eval \times \rho}(H)$ is smaller than a threshold that You deem  to be ``sufficiently small.''  Then Corollary~\ref{cor:strong.valid} offers justification for You to draw the inference that $H$ is false---if it were true, then it would be a ``rare event'' that the IM assigned it such a small upper probability.  

Finally, since
e-processes are likelihood ratios of some sort, the IM output $(\lPi_{z^n}^{\eval \times \rho}, \uPi_{z^n}^{\eval \times \rho})$ depends on the data $z^n$ only through the likelihood function and no other specifics concerning a ``model'' are used for drawing inferences.
This implies that the IM satisfies both the {\em likelihood principle} \citep[e.g.,][]{birnbaum1962, basu1975, bergerwolpert1984} and the strong frequentist-like properties derived above.  This is remarkable because the likelihood principle is Bayesians' turf, but I've shown that both Bayesian-like uncertainty quantification is in reach without sacrificing frequentist-like reliability. 


\subsection{Illustration}
\label{SS:gue}

As a quick illustration, consider inference on the median, $\Theta$, of an underlying distribution; of course, other quantiles can be handled similarly.  Suppose that the observable data are scalars with fully unknown distribution $\prob_\Omega$---that is, the index set $\OO$ is in one-to-one correspondence with the set of distributions on $\RR$.  One way to characterize the mapping from a generic distribution $\prob_\omega$ to the median $\theta$ is through the introduction of a loss function $L_t(z) = |z-t|$, and then defining 
\[ \theta = f(\omega) := \arg\min_t \prob_\omega L_t, \]
the minimizer of the expected absolute difference.  \citet{dey.gue} developed a general e-process construction, extending the universal inference framework of \citet{wasserman.universal} to cases where the quantity of interest is the minimizer of an expected loss.  Their proposed e-process is a sort of ``likelihood ratio,'' where the individual ``likelihood'' terms aren't genuine likelihoods, but suitable quasi-likelihoods constructed from the provided loss function.  Specifically, they define 
\[ \eval_\theta^n = \exp\Bigl[ -\eta \sum_{i=1}^n\{ L_{\hat\theta^{i-1}}(z_i) - L_\theta(z_i)\} \Bigr], \]
where $L_t(z) = |z-t|$ is as above, $\eta > 0$ is a tuning parameter to be discussed further below, and $\hat\theta^k = \hat\theta(z^k) = \text{\tt median}(z_1,\ldots,z_k)$, for $k \geq 1$, with $\hat\theta^0$ any fixed constant.  Then there exists $\eta > 0$ such that $\eval_\theta^n$ is an e-process for the median $\Theta$.  

\citet{dey.gue} describe a data-driven approach for choosing $\eta$ but, for simplicity, I'll take a fixed $\eta$ that depends on a mild assumption about the data-generating process.  By the triangle inequality, the absolute loss function is 1-Lipschitz, i.e., 
\[ \bigl| L_t(z) - L_\theta(z) \bigr| \leq |t - \theta|, \quad \text{all $(z,t,\theta)$}. \]
Then the loss difference $L_t(Z) - L_\theta(Z)$ is a bounded random variable for each pair $(t,\theta)$, hence it's subgaussian and, in particular, it satisfies 
\[ \sup_{\omega: f(\omega) = \theta} \prob_\omega e^{-\eta(L_t - L_\theta)} \leq e^{-\eta \cdot \inf_{\omega: f(\omega)=\theta} \prob_\omega (L_t - L_\theta) + \eta^2 (t - \theta)^2 / 4}. \]
Then the right-hand side is upper-bounded by 1 if and only if
\[ \eta \leq \frac{4 \cdot \inf_{\omega: f(\omega)=\theta} \prob_\omega (L_t - L_\theta)}{(t-\theta)^2}. \]
If $\prob_\omega$ is absolutely continuous, then the expected loss can be approximated locally as
\[ \prob_\omega(L_t - L_\theta) \approx \tfrac12 \, B(\omega) \, (t-\theta)^2, \]
where $B(\omega)$ is the second derivative of $t \mapsto \prob_\omega L_t$ at the median $t=\theta$.  For $L_t(z) = |z-t|$, this second derivative is the density of $\prob_\omega$ evaluated at the median $\theta=f(\omega)$, which is non-negative.  So, if I restrict attention to those (absolutely continuous) $\prob_\omega$'s whose density at $f(\omega)$ is at least $\eps > 0$, then choosing $\eta \leq 2\eps$ will make $\eval_\theta^n$ an e-process.  

For illustration, suppose that Your prior knowledge is an embellished version of the ``95\% sure the median is positive'' belief, i.e., 
\[ q(\theta) = \{0.05 \times 1(\theta < 0) + 1(\theta \geq 0)\} (1 + |\theta|)^{-1}, \quad \theta \in \RR. \]
That is, beyond satisfying $\lprior(\Theta > 0) = 0.95$, posited values of the median further away from 0 are strictly less plausible than values closer to 0; see the gray curve in Figure~\ref{fig:gue}.  Note that only prior knowledge about $\Theta$ is involved here, nothing is said about other features of $\prob_\Omega$.  The other curves in Figure~\ref{fig:gue} represent the unregularized and regularized e-possibility contours---i.e., $\pi_{z^n}^\eval$ and $\pi_{z^n}^{\eval \times \rho}$, respectively---as described above for three different data sets.  Each data set is based on a common sample of size $n=25$ from a Student- distribution with degrees of freedom 2; the true density function at 0 is larger than the threshold $\eps = 0.1$ that I set here for determining $\eta$.  This common data set is shifted so that the sample median is to the left, on top of, and to the right of 0, to represent different degrees of compatibility with prior knowledge.  Here, like in Section~\ref{SS:gains}, the take-away message is that regularization leads to greater efficiency gains when the data shows mild signs of incompatibility with prior knowledge, as in Panel~(a), since those parameter values mildly compatible with the data but at least mildly incompatible with the prior knowledge get downweighted by the e-possibilistic contour. 

\begin{figure}[t]
\begin{center}
\subfigure[$\text{\tt median}(z^n) = -0.5$]{\scalebox{0.6}{\includegraphics{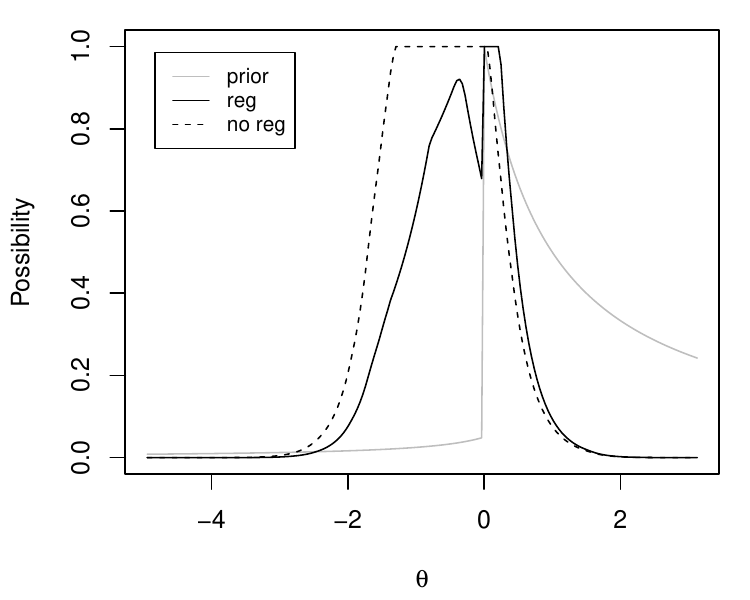}}}
\subfigure[$\text{\tt median}(z^n) = 0$]{\scalebox{0.6}{\includegraphics{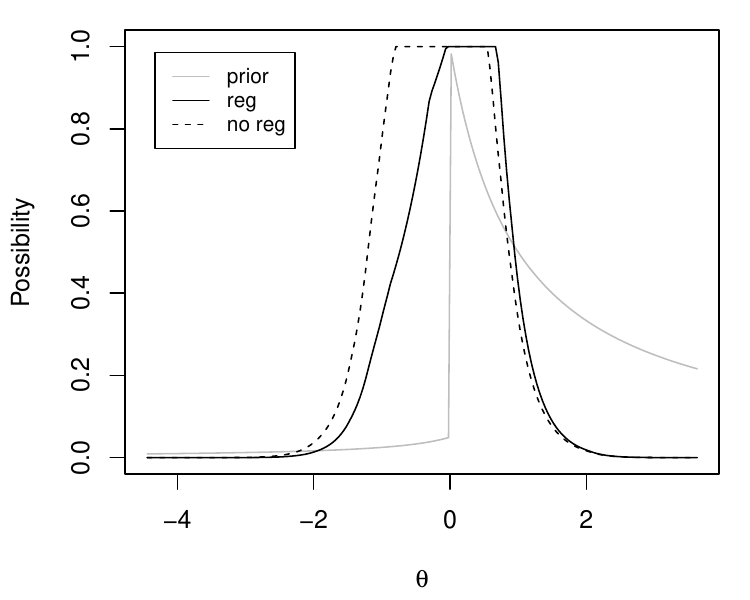}}}
\subfigure[$\text{\tt median}(z^n) = 0.5$]{\scalebox{0.6}{\includegraphics{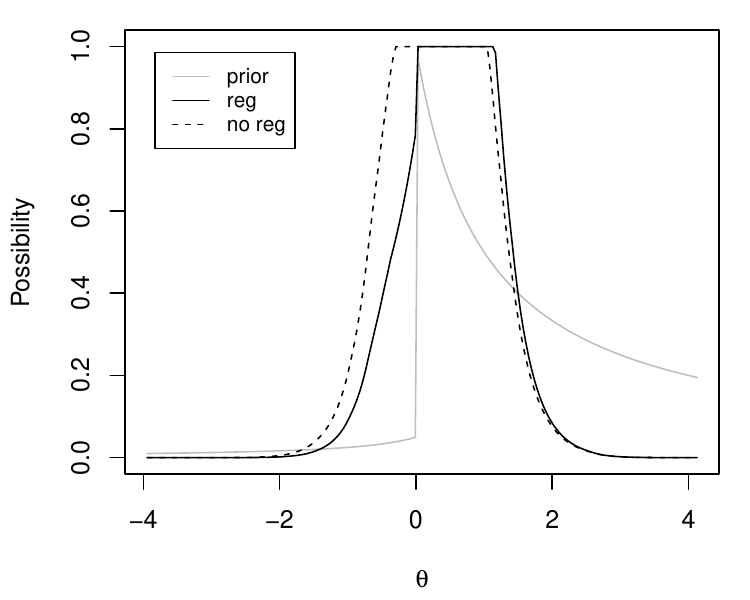}}}
\end{center}
\caption{Plot of the e-possibility contour for three different data sets $z^n$.}
\label{fig:gue}
\end{figure}

\subsection{From uncertainty quantification to behavior}
\label{SS:euq.behavior}

There are (at least) two different behavioral aspects that can be considered.  One pertains to the subjective interpretation of the e-possibilistic IM's output as lower and upper bounds on prices You would assign to gambles.  In that context, a relevant question is if adopting such a pricing scheme put You at risk of being made a sure loser.  The answer is {\em No}---anytime validity implies no sure loss---and these details are in Appendix~\ref{S:euq.coherence}.  

Another behavioral aspect, to be explored in detail here, concerns formal decision-making and the assessment of relevant actions in terms of (lower and upper) expected loss.  For a generic action space $\AA$, let $\ell_a(\theta) \geq 0$ represent the (non-negative) 
loss associated with taking action $a \in \AA$ when the relevant state of the world is $\theta \in \TT$.  Under ideal, perfect-information settings, where the true state of the world $\Theta$ is known, the best decision would correspond to choosing $a^\star$ to minimize the loss $a \mapsto \ell_a(\Theta)$; note that the case of maximizing utility can be recovered from this by defining utility as, say, $u_a = -\ell_a$.  Unfortunately, $\Theta$ is uncertain, so the aforementioned minimization exercise can't be carried out.  But an e-possibilistic IM $(\lPi_{z^n}^{\eval \times \rho}, \uPi_{z^n}^{\eval \times \rho})$ is available for data-driven uncertainty quantification, so I propose to mimic the von Neumann--Morgenstern proposal and evaluate the corresponding upper and lower expected loss:
\[ \uPi_{z^n}^{\eval \times \rho}(\ell_a) = \int_0^1 \Bigl\{ \sup_{\theta: \pi_{z^n}^{\eval \times \rho}(\theta) \geq s} \ell_a(\theta) \Bigr\} \, ds \qquad \text{and} \qquad \lPi_{z^n}^{\eval \times \rho} \, \ell_a = -\uPi_{z^n}^{\eval \times \rho}(-\ell_a). \]
Since the loss is non-negative, the upper and lower expectations surely exist, though they could be $+\infty$ and $-\infty$, respectively, at least for some $a$.  They're both finite at a given $a$ if $\theta \mapsto \ell_a(\theta)$ is previsable with respect to $(\lPi_{z^n}^{\eval \times \rho}, \uPi_{z^n}^{\eval \times \rho})$, e.g., if the loss is bounded.  In any case, these two functionals together determine the IM's assessment of the quality of action $a$.  Since the goal is to make the loss small in some sense, it's reasonable to define the ``best'' action, given data $z^n$, as the minimizer of the upper expected loss:
\begin{equation}
\label{eq:a.hat}
\hat a(z^n) := \arg\min_{a \in \AA} \uPi_{z^n}^{\eval \times \rho}(\ell_a). 
\end{equation}
Further details can be found in \citet{denoeux.decision.2019} and \citet{imdec}.  

For a simple illustration, reconsider the Gaussian example from Section~\ref{SS:gains}, with  
\[ \eval_\theta(z^n) = (nv + 1)^{1/2} \exp\bigl\{ -\tfrac{n}{2} (\theta - \bar z_n)^2 + \tfrac12 ( \tfrac{n}{nv + 1} ) \, \bar z_n^2 \bigr\}, \quad \theta \in \TT, \]
where $v > 0$ is a variance hyperparameter that can take any value; in my numerical examples here and above, I'll take $v=10$.  Starting with vacuous prior information, so that $\rho \equiv 1$, the e-possibilistic IM's contour $\pi_{z^n}^\eval$ is plotted in Figure~\ref{fig:contour.gauss}.  I take the loss function to be squared error: $\ell_a(\theta) = (\theta-a)^2$, also plotted in Figure~\ref{fig:contour.gauss} for several values of $a$.  It's easy to see, based on the symmetry of both the e-process and the loss, that the minimizer of the upper expected loss is $\hat a(z^n) = n^{-1} \sum_{i=1}^n z_i$, the sample mean.  For reasons that will become clear below, it's of interest to evaluate that upper expected loss at the aforementioned minimizer (see Appendix~\ref{A:choquet} for details):
\begin{align*}
\uPi_{z^n}^{\eval} \, \ell_{\hat a(z^n)} & =  \int_0^1 \Bigl\{ \sup_{\theta: \pi_{z^n}^{\eval}(\theta) \geq s} \ell_{\hat a(z^n)}(\theta) \Bigr\} \, ds = n^{-1}\bigl\{ 2 + \log(nv + 1) + (\tfrac{n}{nv+1}) \bar z_n^2 \bigr\}.
\end{align*}

Stick with the same Gaussian illustration, but this time consider regularization corresponding to the partial prior described in Section~\ref{SS:gains}, the strongest of the three partial priors consider there.  Figure~\ref{fig:contour.gauss.reg}(a) shows the possibility contours for the four different values of the partial prior hyperparameter as in Section~\ref{SS:gains}, along with the vacuous prior contour (same as in Figure~\ref{fig:contour.gauss}).  As expected, the regularized contours are a bit more concentrated than the unregularized contour, with more concentration corresponding to the stronger prior knowledge.  From this picture, it's pretty clear that the regularized e-possibilistic IM's risk-minimizing action is closer to 0 than that for the unregularized possibilistic IM.  The introduction of regularization complicates some of the calculations, so I opt for a numerical solution.  Figure~\ref{fig:contour.gauss.reg}(b) shows plots of the e-possibilistic IM's risk assessment as a function of the action $a$, again for different values of the partial prior hyperparameter.  This confirms the intuition that the more the partial prior supports ``$\Theta$ near 0,'' the closer the IM's risk-minimizing rule will be to 0. 

\begin{figure}[t]
\begin{center}
\scalebox{0.6}{\includegraphics{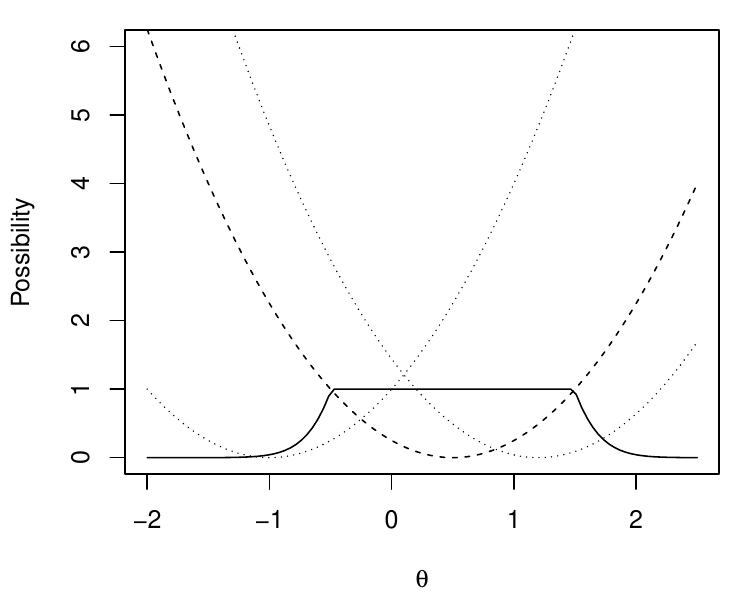}}
\end{center}
\caption{Plot of the (unregularized) e-possibilistic IM's contour (solid), based on data $z^n$ with $n=5$ and $\bar z_n = 0.5$. Plots of the squared error loss function for several different values of $a$ are overlaid (dotted lines), including for $\hat a = \bar z_n$ (dashed line).}
\label{fig:contour.gauss}
\end{figure}

\begin{figure}[t]
\begin{center}
\subfigure[Contour]{\scalebox{0.6}{\includegraphics{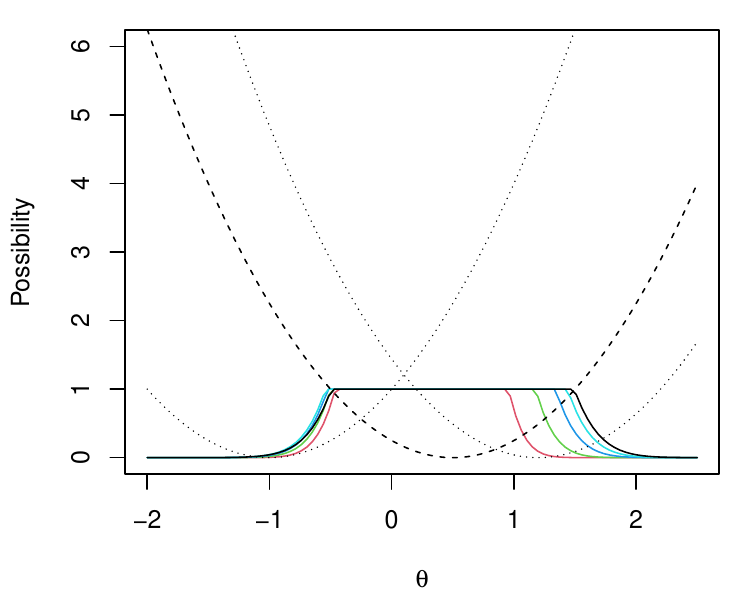}}}
\subfigure[IM's risk assessment]{\scalebox{0.6}{\includegraphics{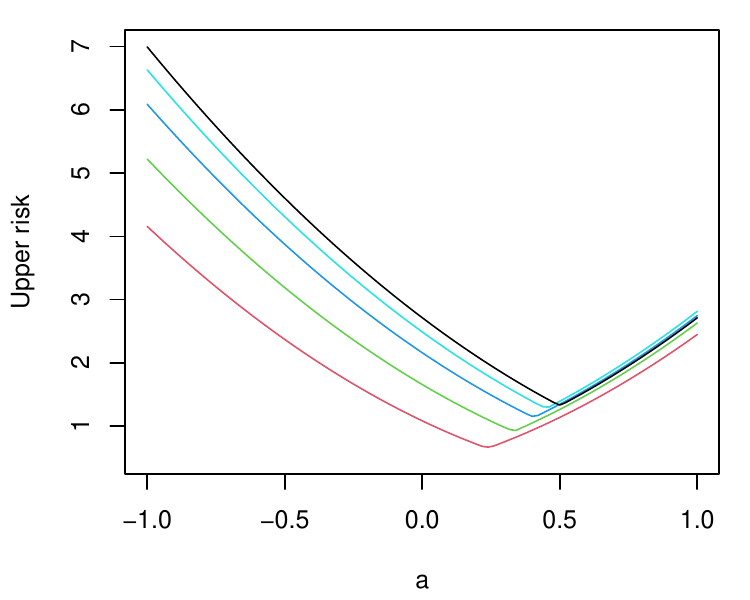}}}
\end{center}
\caption{Panel (a): Plots of the (regularized) e-possibilistic IM's contour, unregularized and regularized, based on data $z^n$ with $n=5$ and $\bar z_n = 0.5$; colors correspond to those in Figure~\ref{fig:gauss}. Plots of the squared error loss function for several different values of $a$ are overlaid (dotted lines), including for $\hat a = \bar z_n$ (dashed line). Panel (b): Plots of the e-possibilistic IM's risk assessment $a \mapsto \uPi_{z^n}^{\eval \times \rho}\,\ell_a$ for five different partial priors---the vacuous prior (black) is minimized at $\bar z_n = 0.5$ while the others shrink toward 0.}
\label{fig:contour.gauss.reg}
\end{figure}

Given the IM's strong reliability properties presented above, it makes sense to ask if these carry over to formal decision-making.  The next result establishes that, indeed, the IM's assessment of an action is expected to not be too optimistic compared to the oracle's. This generalizes a result in \citet{grunwald.epost} for capped e-posteriors.

\begin{thm}
\label{thm:imdec}
Given a loss function $\ell \geq 0$, the the regularized e-possibilistic IM's assessment of actions satisfies 
\begin{equation}
\label{eq:dec.bound}
\uuprob\Bigl\{ \sup_{a \in \AA} \frac{\ell_a(f(\Omega))}{\uPi_{Z^N}^{\eval \times \rho}(\ell_a)} \Bigr\} \leq 1 \quad \text{for all $N$}. 
\end{equation}
In the special case of vacuous partial prior information, \eqref{eq:dec.bound} specializes to 
\[ \sup_{\omega \in \OO} \E_\omega \Bigl\{ \sup_{a \in \AA} \frac{\ell_a(f(\omega))}{\uPi_{Z^N}^{\eval}(\ell_a)} \Bigr\} \leq 1 \quad \text{for all $N$}. \]
\end{thm}

My first remark concerns the interpretation of Theorem~\ref{thm:imdec}: there is no (possibly data-dependent) action $a$ such that the e-possibilistic IM's assessment of $a$ doesn't tend to be overly optimistic (i.e., significantly smaller) than that of an oracle who knows the true state of the world $\Omega$.  To understand this, consider the oracle and IM assessments: $a \mapsto \ell_a(\Theta)$ and $a \mapsto \uPi_{Z^N}^{\eval \times \rho}(\ell_a)$. Imagine a plot of these two functions.  The former, the oracle assessment, will take small values for $a$ near the unattainable ``best'' action $a^\star = \arg\min_a \ell_a(\Theta)$.  Since the data's informativeness is limited, it's unrealistic to expect that the IM's assessment would tend to be small near $a^\star$ too.  Consequently, if the IM's assessment of some $a$ is much more optimistic than the oracle's, then it's likely that this $a$ is different from $a^\star$.  And if the IM's assessment is favoring actions away from $a^\star$, then there'd be a risk of suffering a large loss by taking actions suggested by the IM, e.g., $\hat a$ in \eqref{eq:a.hat}.  Theorem~\ref{thm:imdec} excludes this possibility, so the IM-based assessment is reliable.  

Secondly,the uniform bound on an expectation can be turned into an uniform probability bound.  In particular, 
\begin{equation}
\label{eq:dec.ville}
\uuprob\Bigl\{ \sup_{a \in \AA} \frac{\ell_a(f(\Omega))}{\uPi_{Z^N}^{\eval \times \rho}(\ell_a)} \geq \alpha^{-1} \Bigr\} \leq \alpha \quad \text{for all $\alpha \in [0,1]$, all $N$}. 
\end{equation}
The above says that existence of a triple $(Z^N, \Theta, a)$ such that the IM's risk assessment is drastically smaller than the oracle's is a $\uuprob$-rare event, hence reliability.  

Third, the supremum on the inside of the $\uuprob$-expectation implies that the fixed $a$ can be replaced by any data-dependent $a$, i.e., $a(Z^N)$, leading to 
\[ \uuprob\Bigl\{ \frac{\ell_{a(Z^N)}(f(\Omega))}{\uPi_{Z^N}^{\eval \times \rho}(\ell_{a(Z^N)})} \Bigr\} \leq 1 \quad \text{for all $N$}. \]
In particular, the above holds with IM-based rule $\hat a(Z^N)$ in \eqref{eq:a.hat}.  With this focus on a single, but data-dependent action, say, $\hat a(Z^N)$, there's a bit more that can be said concerning the probability bound \eqref{eq:dec.ville}.  Indeed, \eqref{eq:dec.ville} can be unwrapped as 
\[ \uuprob\Bigl\{ \ell_{\hat a(Z^N)}(f(\Omega)) \geq \alpha^{-1} \uPi_{Z^N}^{\eval \times \rho}\,\ell_{\hat a(Z^N)} \Bigr\} \leq \alpha. \]
That is, it's a $\uuprob$-rare event that the realized loss $\ell_{\hat a(Z^N)}(\Theta)$---corresponding to the uncertain $(Z^N,\Theta)$---incurred by taking the IM's suggested action $\hat a(Z^N)$ is a large multiple of the IM's internal, data-dependent assessment of the risk.

\section{Application}
\label{S:ware}

Statistical, scientific, and ethical challenges emerge when comparing two medical treatments when one is potentially far superior to the other.  The goal is to demonstrably prove that the one treatment is superior, so that the inferior treatment can be safely removed from use, thereby saving/improving lives.  On the one hand, if one treatment is far superior, then, this should be easy to prove statistically with sufficient data.  On the other hand, if one treatment is far superior, then it's borderline unethical
to continue randomizing patients to the inferior treatment just to get ``sufficient data.'' 
This context clearly highlights the importance of (a)~accommodating adaptive data-collection schemes and (b)~using reliable and efficient statistical methods from which justifiable conclusions can be drawn while minimizing patients' exposure to an inferior treatment.  

One well-known study is presented in \citet{ware1989}.  His investigation concerned persistent pulmonary hypertension in newborns.  The standard treatment for many years, called the conventional medical therapy (CMT), had a mortality rate of at least 80\%.  But by 1985, a new potential treatment had emerged, namely, extracorporeal membrane oxygenation (ECMO), for which reported survival rates were at least 80\%.  Unfortunately, the empirical support for this latter claim was rather thin: only one randomized trial had been performed and, in that trial, only one patient was assigned to CMT.  Despite ECMO's strong performance in these studies, the combination of its statistically-limited support and its potential for side-effects gave some medical researchers pause.  Ware's paper describes a trial that's based on randomization (within blocks) until a prespecified number of deaths---namely, 4---are observed in {\em each group}.  Ware's data is as follows:
\begin{equation}
\label{eq:ware.data}
\text{CMT: 10 patients with 4 deaths} \qquad \text{ECMO: 9 patients with 0 deaths}. 
\end{equation}
It's important to note that the investigators {\em didn't carry out the design as planned}, i.e., they didn't continue following patients until a fourth death in the ECMO group was observed.  This emphasizes the need for anytime valid statistical procedures. 

Assuming independence, and that the uncertain survival probabilities---denoted by $\Theta_\text{cmt}$ and $\Theta_\text{ecmo}$---are constant across patients, the likelihood function based on \eqref{eq:ware.data} is 
\[ L_{z^n}(\theta) \propto \theta_\text{cmt}^6 \, (1 - \theta_\text{cmt})^4 \, \theta_\text{ecmo}^9, \qquad \theta=(\theta_\text{cmt}, \theta_\text{ecmo}), \]
where ``$z^n$'' is just the symbol I'll use for the data in \eqref{eq:ware.data}.  For the e-process, I'll use a slightly modified version of the recommendation in \citet{turner.grunwald.2023},
\[ \eval_\theta(z^n) = \frac{\hat\theta_{\text{cmt},\beta}^6 \, (1 - \hat\theta_{\text{cmt},\beta})^4 \, \hat\theta_{\text{ecmo},\beta}^9}{\theta_\text{cmt}^6 \, (1 - \theta_\text{cmt})^4 \, \theta_\text{ecmo}^9}, \qquad \theta = (\theta_\text{cmt}, \theta_\text{ecmo}), \]
where 
\[ \hat\theta_{\text{cmt},\beta} = \frac{6 + \beta}{10 + 2\beta} \quad \text{and} \quad \hat\theta_{\text{ecmo},\beta} = \frac{9 + \beta}{9 + 2\beta}, \]
with $\beta=0.18$ as suggested in \citet[][Sec.~3]{turner.grunwald.2023}. The results of my analysis based on this e-process will be presented below; see, e.g., Figure~\ref{fig:ware.2d}(a). 

The reader can glean from the above discussion that, while the data in this particular study might be rather limited, there is relevant ``prior information'' available.  But what to do with it?  The frequentist analysis in \citet{ware1989} formally ignores prior information; informally, however, the prior information is used in ad hoc ways.  Bayesian solutions face the problem that prior information is insufficient to determine a precise prior distribution for $\Theta$.  A common strategy in such cases is a prior sensitivity analysis based on a few selected priors that are ``consistent'' with the information available \citep{kass.greenhouse.1989, kass1992}.  Clearly, none of these priors are ``right,'' so the best one can hope for is that the sensitivity analysis reveals that the posterior isn't sensitive to the prior.  
\citet[][Sec.~5]{walley1996} offers a generalized Bayesian analysis that also virtually ignores the prior information; but instead of literally ignoring it, he models the near-ignorance by a prior credal set that contains all independent beta distributions for $(\Theta_\text{cmt}, \Theta_\text{ecmo})$.  In cases like the present CMT/ECMO study, where data is limited, my goal is to leverage, rather than ignore, real-but-necessarily-incomplete prior information in order to {\em strengthen} the analysis and justification for its conclusions.  

Next, I'll present my encoding of the available partial prior information---with minimal embellishment on the information given in \citet{ware1989}---as a possibility distribution to be used in my subsequent analysis.  To be clear, I have no medical expertise; so, my analysis is sure to be overly simplistic and, hence, is only for illustrative purposes.  That said, I think my formulation and conclusions drawn are quite reasonable.  

There are two statements---``$\leq 0.2$'' and ``$\geq 0.8$''---that stand out in Ware's report, which are best interpreted as ``prior limits'' for $\Theta_\text{cmt}$ and $\Theta_\text{ecmo}$, respectively.  What's missing are quantitative statements about the confidence or degree of belief in these limits, but qualitative statements are made in the text.  This is where some subjective judgment becomes necessary.  As Ware explains, there's reasonably strong support for the prior limit ``$\Theta_\text{ecmo} \geq 0.8$'' based on historical data.  Importantly, Ware argues that the study inclusion protocol, etc.~are such that the patients in these previous studies and those in his study are more-or-less homogeneous.  The support for the prior limit ``$\Theta_\text{cmt} \leq 0.2$,'' on the other hand, is considerably weaker: there was only one case involving a patient that could've received ECMO but was randomly assigned to CMT instead.
Based on my interpretation of Ware's report, I assign confidence as follows:
\begin{itemize}
\item 90\% confident in ``$\Theta_\text{ecmo} \geq 0.8$'' and
\vspace{-2mm}
\item 50\% confident in ``$\Theta_\text{cmt} \leq 0.3$.'' 
\end{itemize}
Aside from the difference in confidence levels, which I'll explain shortly, notice that I stretch out the limit for $\Theta_\text{cmt}$ a bit; this is because Ware used ``0.2'' as a sort of prior mode for $\Theta_\text{cmt}$, so it might be prudent to have the mode closer to the middle of the range to which I assign some non-trivial degree of confidence.  The confidence levels above are based on my judgment that researchers are quite confident in the performance of ECMO but far less---say, roughly {\em half} as---confident about CMT on similar populations of patients.  Mathematically, I opt to encode this vaguely-stated quantification of uncertainty as a possibility measure for $\Theta=(\Theta_\text{cmt}, \Theta_\text{ecmo})$.  Marginally, the two contours are
\begin{align*}
q_\text{ecmo}(\theta_\text{ecmo}) & = 0.1 + 0.9 \cdot 1(\theta_\text{ecmo} \geq 0.8) \\
q_\text{cmt}(\theta_\text{cmt}) & = 0.5 + 0.5 \cdot 1(\theta_\text{cmt} \leq 0.3).
\end{align*}
The aforementioned ``confidence levels'' correspond to the properties 
\[ \uprior_\text{cmt}(\Theta_\text{cmt} \leq 0.3) = 0.5 \quad \text{and} \quad \lprior_\text{ecmo}(\Theta_\text{ecmo} \geq 0.8) = 0.9. \]
Following Walley and others, I'll treat $\Theta_\text{cmt}$ and $\Theta_\text{ecmo}$ as independent {\em a priori} and then take the joint possibility measure $\uprior$ for $\Theta=(\Theta_\text{cmt}, \Theta_\text{ecmo})$ to have contour $q$ equal to the product of $q_\text{ecmo}$ and $q_\text{cmt}$ above. This $q$ is a piecewise constant function on the unit square, taking value 1 on the rectangle $[0,0.3] \times [0.8, 1]$ and smaller values in the three other rectangles.  While others might disagree to some extent with the particular confidence levels that I chose, I believe that this possibilistic prior is consistent with the information Ware presented.  In particular, the credal set determined by $\uprior$ contains independent products of certain beta distributions, among others.  

Plots of the unregularized and regularized e-processes are shown in Figure~\ref{fig:ware.2d}; for the regularized e-process, I'm using the same calibrator $\gamma$ as described in Section~\ref{SS:possibilistic.case}.  As expected, the e-process contour bottoms out at the value corresponding to the simple sample proportions, namely, $(0.6, 1)$.  
Note that the unregularized e-process contours are smooth, whereas those of the regularized version's are rough in some places; this roughness is due to the relatively large jump discontinuity in the prior contour there.  The heavy red line marks the confidence sets $C_\alpha$ and $C_\alpha^\text{reg}$ for $\alpha=0.05$, and there are two notable observations.  First, as it pertains to $\Theta_\text{ecmo}$, for which prior information is relatively strong, the limits are much tighter in the regularized case compared to the unregularized.  Second, for $\Theta_\text{cmt}$, the limits are a bit looser for the regularized case compared to the unregularized.  The latter point might seem disappointing, but this is exactly what should happen: prior knowledge that's somewhat incompatible with data ought to result in more conservative inference.  

\begin{figure}[t]
\begin{center}
\subfigure[Unregularized]{\scalebox{0.55}{\includegraphics{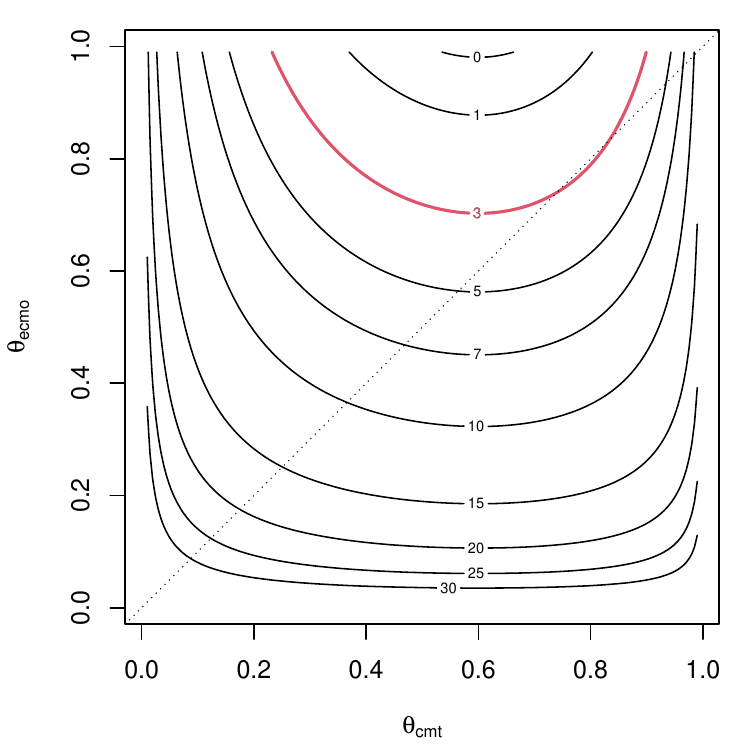}}}
\subfigure[Regularized]{\scalebox{0.55}{\includegraphics{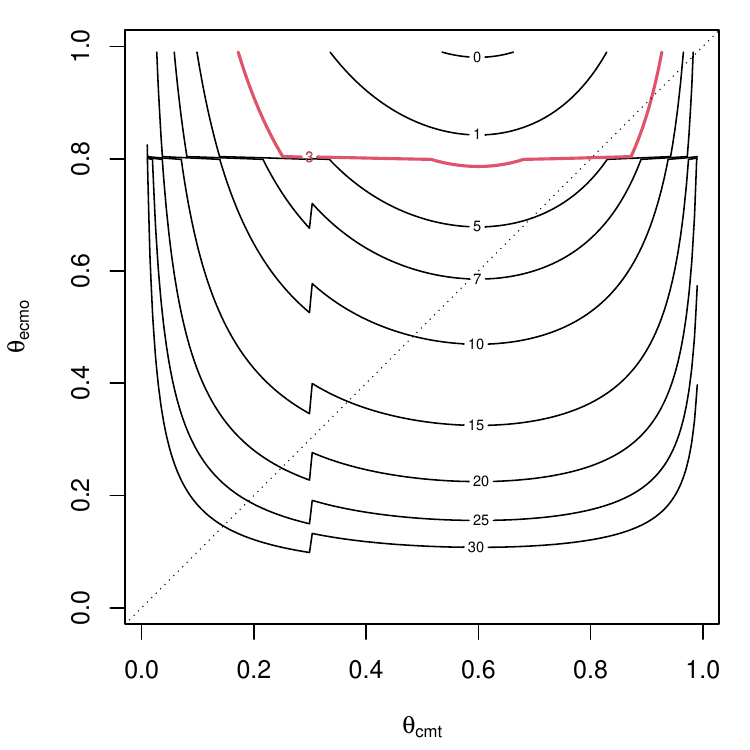}}}
\end{center}
\caption{Plots of the unregularized and regularized e-processes based on Ware's CMT/ECMO data.  Heavy red lines mark the corresponding 95\% confidence sets.}
\label{fig:ware.2d}
\end{figure}

A relevant question is if EMCO is more effective than CMT.  One way to answer this is to test the hypothesis ``$\Theta_\text{ecmo} \leq \Theta_\text{cmt}$'' versus ``$\Theta_\text{ecmo} > \Theta_\text{cmt}$.'' The diagonal line through the two plots in Figure~\ref{fig:ware.2d} represents the boundary between these two propositions and, since the red contour curves intersect with the diagonal line, the corresponding e-process-based tests cannot reject the hypothesis ``$\Theta_\text{ecmo} \leq \Theta_\text{cmt}$'' at level $\alpha=0.05$.
For comparison, Ware presents an analysis that using Fisher's exact test, leading to a p-value of 0.054, which is consistent with my conclusions based on Figure~\ref{fig:ware.2d}(a).  My results are based on procedures proved to be anytime valid, so this provides additional comfort given that Fisher's exact test is not anytime valid.

For uncertainty quantification, the possibility contours, $\pi^\eval$ and $\pi^{\eval \times \rho}$, in this case look identical to those in Figure~\ref{fig:ware.2d}, just with different numerical labels on the contours.  So, the confidence regions and test conclusions derived from $\pi^\eval$ and $\pi^{\eval \times \rho}$ give exactly the same results as those obtained from Figure~\ref{fig:ware.2d}.  But there's more that can be done with the IM's possibilistic uncertainty quantification.  First, consider a feature $\Delta = \Theta_\text{ecmo} - \Theta_\text{cmt}$, the difference in survival rates between EMCO and CMT.  The possibility-theoretic {\em extension principle} 
determines a (marginal) possibility contour 
for $\Delta$ from that for $\Theta$:
\[ \phi_{z^n}^\eval(\delta) = \sup_{\theta \in [0,1]^2: \, \theta_\text{ecmo}-\theta_\text{cmt}=\delta} \pi_{z^n}^\eval(\theta), \quad \delta \in [-1,1]. \]
The regularized e-possibilistic marginal IM contour $\phi_{z^n}^{\eval \times \rho}$ is defined analogously. The corresponding upper probabilities $\uPhi_{z^n}^\eval$ and $\uPhi_{z^n}^{\eval \times \rho}$ are defined via optimization as usual.  The dashed line at $\alpha=0.05$ determines a marginal anytime valid confidence interval for $\Delta$ and, at least for the regularized version (in red), this more-or-less agrees with the generalized Bayes 95\% credible interval presented in \citet[][Fig.~1]{walley1996}.  The message from this plot is that, while ``$\Delta > 0$'' is highly possible, nearly the entire range $[-1,1]$ is ``sufficiently possible'' based on Ware's data.  Admittedly, the extension principle is conservative, but it's the most direct way to marginalize while preserving validity.

\begin{figure}[t]
\begin{center}
\scalebox{0.65}{\includegraphics{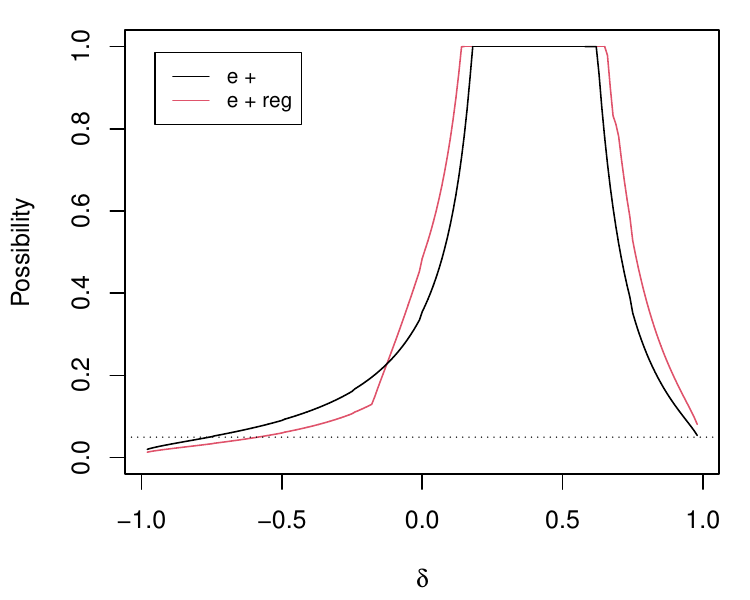}}
\end{center}
\caption{Plots of the regularized and unregularized e-possibilistic IM's marginal contour for $\Delta = \Theta_\text{ecmo} - \Theta_\text{cmt}$ based on Ware's CMT/ECMO data.}
\label{fig:ware.1d}
\end{figure}

A formal decision-theoretic approach can be considered, using suitable lower and upper expected loss/utility based on the e-possibilistic IM output.  \citet{walley1996} argues that of utmost importance---or {\em utility}---is a patient's survival.  That is, the desired goal is for the patient to survive with the treatment they were given, all other factors are irrelevant.  So, if $Y$ is a binary indicator that the patient in question survives, then the utilities associated with CMT and ECMO are 
\[ u_\text{cmt}(y) = u_\text{ecmo}(y) = y, \quad y \in \{0,1\} \equiv \{\text{dies}, \text{survives}\}. \]
In this case, the expected utilities are $\Theta_\text{cmt}$ and $\Theta_\text{ecmo}$.  Since $(\Theta_\text{cmt}, \Theta_\text{ecmo})$ is uncertain, I can't simply pick which of CMT and ECMO has higher expected utility.  But I can evaluate the lower and upper expectation of the difference between the two expected utilities, namely, $\Theta_\text{ecmo} - \Theta_\text{cmt}$, based on the full uncertainty quantification provided by the e-possibilistic IM.  I've already obtained the marginal IM for the difference $\Delta$ in expected utilities, so I just need to evaluate the lower and upper expectations.  For the unregularized e-possibilistic IM, the upper and lower expectation of $\Delta$ are 
\begin{align*}
\uPhi_{z^n}^\eval(\Delta) = \int_0^1 \sup\{ \delta: \phi_{z^n}^\eval(\delta) > s\} \, ds \quad \text{and} \quad \lPhi_{z^n}^\eval(\Delta) = -\uPhi_{z^n}^\eval(-\Delta), 
\end{align*}
and analogously for the regularized version.  Numerically, I get:
\[ \bigl[ \lPhi_{z^n}^\eval(\Delta), \uPhi_{z^n}^\eval(\Delta) \bigr] = [-0.042, 0.739] \quad \text{and} \quad \bigl[ \lPhi_{z^n}^{\eval \times \rho}(\Delta), \uPhi_{z^n}^{\eval \times \rho}(\Delta) \bigr] = [0.086, 0.820]. \]
That the first, unregularized interval contains 0 means that I can't rule out, based on Ware's data, that CMT is preferred to ECMO in terms of utility.  The regularized interval, however, is strictly to the right of the origin, which implies that ECMO is the preferred treatment in terms of expected utility, albeit just barely.  For comparison, Walley reports his version the ``lower and upper expectation of $\Delta$'' as $[0.152, 0.5]$, so he too concludes that ECMO is the demonstrably better treatment.  

\section{Conclusion}
\label{S:discuss}

This paper develops the new concept of and theory associated with {\em regularized e-processes}, from which I further develop reliable---i.e., anytime valid and efficient---inference and uncertainty quantification.  On the technical side, the regularized e-process involves the combination of a fully data-driven e-process and prior knowledge about $\Theta$ that You, the investigator, might have.  Importantly, Your prior information need not be---and typically won't be---sufficiently complete to pinpoint a single prior as a Bayesian analysis would require, so my proposal explicitly draws on aspects of imprecise probability theory.  My proposed regularization discounts those values of $\Theta$ that are incompatible with the available prior knowledge, making it easier to ``reject'' such values compared to the purely data-driven e-process, hence the efficiency gains.  The critical point, however, is that any such non-trivial discounting like described above jeopardizes the original e-process's inherent anytime reliability; therefore, the proposed regularization requires care.  

By encoding the partial prior information rigorously as a coherent imprecise probability, I can build a collection of joint distributions for $(\text{data}, \Theta)$ compatible with the assumed model and available prior information.  Then I generalize the now-familiar anytime validity property in a sound way, accounting for the available prior information, and similarly generalize Ville's inequality to prove that the proposed regularized e-process remains anytime valid in this more general sense.  This reformulation offers mathematical justification for You to accept the efficiency gains offered by regularization.

A number of interesting and challenging questions remain open, I'll end here with a few remarks about two of them.  First, to put the theory presented here into practice, You must convert what You know about the uncertain $\Theta$ into an imprecise probability.  For example, in the clinical trial data illustration of Section~\ref{S:ware}, some degree of belief attached to bounds on $\Theta$ was converted into a possibility contour.  The point is that You need to know (a)~how to tease relevant quantitative details out of Your mostly qualitative corpus of knowledge and (b)~how to properly map those relevant details to a suitable imprecise probability.  On top of this, the second point, is that a possibilistic formulation of the prior information requires choice of a calibrator as I described in Section~\ref{SS:possibilistic.case}.  The elicitation and encoding of partial prior information is an interesting question, more psychological than statistical, something that I plan to explore.  A thorough comparison of how different calibrators perform in this context across different model/data settings is also needed. 


Second, my illustrations only considered cases involving $\Theta$ of fixed dimension.  For applications to high-dimensional problems, it's important to recognize that an efficiency gain can be realized only if the regularization adapts to the dimension of $\Theta$.  This is typically accommodated by tying the dimension and sample size together and then allowing the penalty/prior to depend explicitly on the sample size.  This is a little awkward in the present context, for several reasons.  One is that the data are streaming, and the stopping rule is dynamic, so if dimension and sample size are linked, then it's as if the inference problem doesn't have a well-defined target.  Another challenge is dealing with the fact that $\credal$ is assumed to be a {\em real} representation of Your prior knowledge and it might seem strange that Your knowledge depends on dimension.  The intuition, however, behind standard penalties/priors---e.g., ``most of the entries in $\Theta$ are near-zero''---is inherently imprecise and notions like ``most'' are relative to the number of things in question, which in this case is the dimension of $\Theta$.  This seems doable, but the question remains: specifically how should $\credal$ adapt to dimension?

\section*{Acknowledgments}

This work is partially supported by the U.S.~National Science Foundation, grants SES--2051225 and DMS--2412628.



\appendix

\section{Technical remarks}
\label{A:remarks}

Below are several remarks offering further insights on and explanation of some specific points made in the main text. 

\begin{remark}
\label{re:continuity}
If one encounters a rare case in which the mapping from $\Omega$ to $\Theta=f(\Omega)$ is discontinuous, then there's a simple work-around.  Just carry out the developments below as if $\Omega$ is the quantity of interest, and then marginalize down to $\Theta$ at the end, e.g., by stating the relevant hypotheses about $\Theta$ in terms of the corresponding hypotheses about $\Omega$.  All of the theory presented in the paper supports such an approach. The only downside is that carrying the higher-dimensionality of $\Omega$ from beginning to (near) end can cause efficiency to be lost compared to dropping the dimension to that of $\Theta$ at the outset. 
\end{remark}

\begin{remark}
\label{re:betting}
Following de Finetti, it is common to interpret lower and upper probabilities $(\lprior,\uprior)$ in terms of buying/selling prices for gambles:
\begin{equation}
\label{eq:prices}
\begin{split}
\lprior(H) & = \text{Your supremum buying price for $\$1(\Theta \in H)$} \\
\uprior(H) & = \text{Your infimum selling price for $\$1(\Theta \in H)$}.
\end{split}
\end{equation}
That is, You'd be willing to buy a ticket that pays \$1 if ``$\Theta \in H$'' from me for no more than $\$\lprior(H)$ and, similarly, You'd be willing to sell a ticket to me that pays \$1 if ``$\Theta \in H$'' to me for no less than $\$\uprior(H)$.  Intuitively, You shouldn't be willing to buy a gamble for a price higher than You're willing to sell it for---that would put You at risk of a sure loss---so the setup would be unsatisfactory if it didn't rule out this case.  Indeed, the assumed structure of the credal set implies that $\lprior(H) \leq \uprior(H)$ for all $H$.  Moreover, $\uprior$ is not just a summary or derivative of $\credal$, the two are in one-to-one correspondence:
\begin{equation}
\label{eq:credal}
\credal = \{\prior: \prior(H) \leq \uprior(H) \text{ for all measurable $H$} \}. 
\end{equation}
That is, $\uprior$ is exactly the upper envelope of $\credal$ and, consequently, credal set-driven imprecise probabilities are {\em coherent} \citep[e.g.,][]{walley1991}, generalizing the notion put forth by de Finetti, Savage, Ramsey, and others for precise probability.
\end{remark}

\begin{remark}
\label{re:vac.supremum}
The simplification in \eqref{eq:vac.supremum} may not be immediately obvious, so here's an explanation.  Expectation of a random variable is linear in the distribution and, therefore, it's supremum over a closed and convex set of distributions is attained at the extremes, on the boundary.  If the prior is vacuous, so that the credal set contains all distributions, then the boundary consists of point mass distributions. Hence, the supremum expectation over all distributions is attained at a point mass distribution.
\end{remark}

\begin{remark}
\label{re:prob.to.poss}
The choice of prior contour in Section~\ref{SS:gains} is an application of the general {\em probability-to-possibility transform} described in, e.g., \citet{dubois.etal.2004}, \citet{hose.hanss.2021}, and \citet{hose2022thesis}.  Mathematically, the transformation is rather straightforward.  If $Y$ is a random vector taking values in $\YY$ with distribution $\prob_Y$ and corresponding density/mass function $f_Y$, then the probability-to-possibility transform results in a possibility measure with contour defined via  
\[ \psi(y) = \prob_Y\{ f(Y) \leq f(y) \}, \quad y \in \YY. \]
That $\psi$ is a possibility contour is easy to see: if $f_Y$ has mode $y^\text{mode}$, then $\psi(y^\text{mode})=1$; the unbounded-$f_Y$ case can be handled similarly.  In Section~\ref{SS:gains} of the main text, the contour that involved the chi-square distribution function is obtained by applying the above formula with $f_Y$ a Gaussian density.  What's notable about this particular possibility measure construction is that the credal set corresponding to $\psi$ is the smallest of all those credal sets that contain $\prob_Y$ and whose upper envelope is a possibility measure.  So, if, like in Section~\ref{SS:gains}, ones prior knowledge consists of only a surprise-assessment that agrees with a Gaussian probability, then the choice of prior possibility contour recommended there is the best choice in the sense that it's corresponds to the smallest (possibilistic) credal set consistent with the available prior knowledge.  
\end{remark}

\begin{remark}
\label{re:why.possibility}
A reasonable question is: why are {\em possibilistic} IMs appropriate?  To me, the most compelling justification comes from the uniform validity result in Corollary~\ref{cor:uniform.vac}.  The equivalence---see Equation~\eqref{eq:equivalence} below---in the proof holds for all imprecise-probabilistic IMs, but what I'm calling the ``contour'' $\pi^\eval$ has properties under a possibilistic formulation that it doesn't have under other formulations.  For a generic, data-dependent upper probability $\uPi_{Z^N}$ on $\TT$, the function $\theta \mapsto \uPi_{Z^N}(\{\theta\})$ doesn't completely determine $\uPi_{Z^N}$ and, moreover, it can happen that $\theta \mapsto \uPi_{Z^N}(\{\theta\})$ is always small, no matter what $Z^N$ is.  Then the probability (with respect to the sampling distribution of $Z^N$) that it's less than some $\alpha < 1$ could be large---perhaps even equal to 1.  It's unique to the possibilistic framework that the contour fully determines the upper probability and, moreover, takes values arbitrarily close to 1.  Without this special structure, strong validity and, hence, uniform validity can't be attained.  If it could be attained by some other, non-possibilistic IM construction, then Lemma~1 in \citet{martin.partial2} says that there's a possibilistic IM that's no worse in terms of efficiency.  The conclusion is that, if strong validity and the safety it offers is a priority, which it is to me, then the possibilistic formulation is without loss of generality/efficiency.  
\end{remark}

\begin{remark}
\label{re:typically}
The function $\pi_{z^n}^\eval$ fails to be a possibility contour at a given $z^n$ if and only if $\eval_\theta(z^n)$ is strictly greater than 1 for all $\theta$.  But e-processes have expected value upper-bounded by 1, so it'd be exceptionally rare, although not impossible, for $z^n$ to not be particularly compatible with any $\theta$, so that $\eval_\theta(z^n)$ is everywhere greater than 1 and, hence, $\pi_{z^n}^\eval$ is everywhere below 1. This doesn't affect the statistical properties, only the interpretation of the IM; see Appendix~\ref{S:euq.coherence} below. 
\end{remark}

\section{Choquet integration}
\label{A:choquet}


Choquet integration has an important role to play in the main paper's developments.  Basically, Choquet integration plays the same role in imprecise probability theory as Lebesgue integration does in ordinary/precise probability theory.  That is, just as Lebesgue integration is the go-to mathematical framework for defining expectation with respect to probability measures, Choquet integration is the appropriate way to extend lower/upper probabilities, at least for the kind of models in consideration here, to more general lower/upper expectations.  The path to making this connection isn't exactly direct, and the details are too involved to present here, but I think a relatively brief overview would be beneficial.  My summary here is based primarily on details presented much more thoroughly and rigorously in \citet{lower.previsions.book}.  

Let $g: \TT \to \RR$ be a function, which I'll assume to be non-negative only for simplicity; if $g$ can take both positive and negative values, then one apply the developments here to the difference between the positive and negative parts of $g$.  If $\uprior$ is a general capacity---a normalized, monotone set function---supported on subsets of $\TT$, then the Choquet integral of $g$ with respect to $\uprior$ is defined as 
\begin{equation}
\label{eq:choquet}
\mathcal{I}_\text{\sc choq}(g) := \int_0^\infty \uprior\{ \theta \in \TT: g(\theta) \geq t \} \, dt, 
\end{equation} 
where ``$\int$'' on the right-hand side is a Riemann integral, which is well-defined since the integrand is a monotone non-increasing function of $t$.  In some sense, there's nothing particularly special about defining an ``integral''---anyone can do it.  The challenge is defining an integral that represents something relevant.  In the present context, the most meaningful notion of an expected value of $g(\Theta)$ with respect to a coherent upper probability $\uprior$ (defined on all subsets $H$ of $\TT$) is the upper envelope 
\[ \uprior\,g = \sup_{\prior \in \credal} \E^{\Theta \sim \prior}\{ g(\Theta) \}, \]
which is the second expression given in Equation~(6) of the main paper; recall that $\credal$ here is the set of all probabilities dominated by $\uprior$, as in Equation~(5).  How is this connected to the Choquet integral?

What links the Choquet integral above to the upper expectation is a deep result of Walley's, concerning the so-called {\em natural extension} of $\uprior$ from an upper probability to an upper expectation.  In a purely mathematical sense, one may have a function $f$ defined on a domain $\XX$ with certain properties, and the relevant question is if $f$ can be extended from $\XX$ to a function $f^\star$ defined on a larger domain $\XX^\star$, such that there's agreement on $\XX$, i.e., $f^\star(x) = f(x)$ for $x \in \XX$, and $f^\star$ maintains $f$'s relevant properties on $\XX^\star \setminus \XX$.  In the present context, the upper envelope can be viewed as a functional $1_H \mapsto \uprior\,1_H := \uprior(H)$ defined on the collection $\{\theta \mapsto 1_H(\theta): H \subseteq \TT\}$ of indicator functions/gambles.  This functional has a coherence property, by assumption, so the question is if it can be extended to a broader class of (bounded\footnote{Walley's developments focus on bounded gambles, but Part~II of \citet{lower.previsions.book} generalizes Walley's results to certain unbounded gambles.}) gambles without sacrificing coherence.  \citet[][Ch.~3]{walley1991} answers this question in the affirmative, with what he calls the {\em natural extension}---the extension that imposes the least additional structure while preserving coherence.  On the importance of natural extension, \citet[][p.~121--122]{walley1991} writes:
\begin{quote}
{\em ...natural extension may be seen as the basic constructive step in statistical reasoning; it enables us to construct new previsions from old ones.}
\end{quote}
The formula for the natural extension is rather complicated and not necessary for the present purposes.  The relevant point here is that the {\em upper envelope theorem}\footnote{As the title of their book suggests, \citet{lower.previsions.book} focus almost exclusively on {\em lower} previsions, and what they prove is a {\em lower} envelope theorem.  There is, however, an analogous result for the upper prevision and that's what I'm referring to here as the {\em upper envelope theorem}.} in \citet[][Theorem~4.38]{lower.previsions.book} links the upper expectation $\uprior\,g$ to Walley's natural extension of $\uprior$ to (bounded) gambles.  Then they follow up \citep[][Theorem~6.14]{lower.previsions.book} by linking the natural extension of $\uprior$ to the Choquet integral in \eqref{eq:choquet}.  Therefore, the ``$\mathcal{I}_\text{\sc choq}(g)$'' notation can be dropped---the Choquet integral and the upper expectation $\uprior\,g$ are the same, so the latter notation is sufficient.

Then the formula given in Equation~(8) for the upper expectation with respect to a possibility measure $\uprior$ determined by contour $q$ follows immediately---or at least {\em almost} immediately.  Using the definition $\uprior(H)$ via optimization as in Equation~(7), the Choquet integral formula \eqref{eq:choquet} above reduces to
\[ \uprior\,g = \int_0^\infty \Bigl\{ \sup_{\theta: g(\theta) \geq t} q(\theta) \Bigr\} \, dt. \]
This expression looks similar to the formula given in Equation~(8), but it's not the same; the latter roughly has the roles of $g$ and $q$ in the above expression reversed.  This apparent ``symmetry'' in the roles of $g$ and $q$, and the corresponding alternative form of the Choquet integral as I advertised in Equation~(8), is established in Proposition~7.14 (and Proposition~C.8) of \citet{lower.previsions.book} for the case of bounded $g$; this is generalized to certain unbounded gambles $g$ in, e.g., their Proposition~15.42. 

For a bit of practice with the possibility-theoretic Choquet integral, I'll first demonstrate that Choquet integration formula in Equation~(8) for a possibility measure $\uprior$ reduces to the definition of upper probability in Equation~(7), via optimization of the contour $q$, when the function $g$ is an indicator, i.e., $g(\theta) = 1(\theta \in H)$ for some $H \subseteq \TT$.  For such a case, the integrand in Equation~(8) is given by 
\[ s \mapsto \sup_{\theta: q(\theta) \geq s} 1(\theta \in H) = \begin{cases} 1 & \text{if $s \leq \sup_{\theta \in H} q(\theta)$} \\ 0 & \text{otherwise}. \end{cases} \]
Then it's clear that 
\begin{align*}
\uprior\,g & = \int_0^1 \sup_{\theta: q(\theta) > s} 1(\theta \in H) \, ds \\
& = \int_0^{\sup_{\theta \in H} q(\theta)} 1 \, ds + \int_{\sup_{\theta \in H} q(\theta)}^1 0 \, ds \\
& = \sup_{\theta \in H} q(\theta) \\
& = \uprior(H), 
\end{align*}
as was to be shown.  Warm-up complete. 

Next, I have a slightly more ambitious goal of verifying the formula below that was presented without proof in Section~4.5.2 of the main paper:
\begin{align*}
\uPi_{z^n}^{\eval} \, \ell_{\hat a(z^n)} =  \int_0^1 \Bigl\{ \sup_{\theta: \pi_{z^n}^{\eval}(\theta) \geq s} \ell_{\hat a(z^n)}(\theta) \Bigr\} \, ds = n^{-1}\bigl\{ 2 + \log(nv + 1) + (\tfrac{n}{nv+1}) \bar z_n^2 \bigr\}.
\end{align*}
In this case, the contour is $\pi_{z^n}^\eval(\theta) = 1 \wedge \eval_\theta(z^n)^{-1}$, where 
\[ \eval_\theta(z^n) = (nv + 1)^{-1/2} \exp\bigl\{ \tfrac{n}{2} (\theta - \bar z_n)^2 - \tfrac12 ( \tfrac{n}{nv + 1} ) \, \bar z_n^2 \bigr\}, \quad \theta \in \RR. \]
Then 
\begin{align*}
\pi_{z^n}^\eval(\theta) \geq s & \iff 1 \wedge \eval_\theta(z^n)^{-1} \geq s \\
& \iff \eval_\theta(z^n) \leq s^{-1} \\
& \iff (\theta - \bar z_n)^2 \leq n^{-1} \bigl\{ 2\log(s^{-1}) + \log(nv+1) + (\tfrac{n}{nv+1}) \bar z_n^2 \bigr\}.
\end{align*}
It just so happens that $\ell_{\hat a(z^n)}(\theta) = (\theta - \bar z_n)^2$, and from this it's clear that 
\[ \sup_{\theta: \pi_{z^n}^{\eval}(\theta) \geq s} \ell_{\hat a(z^n)}(\theta) = n^{-1} \bigl\{ 2\log(s^{-1}) + \log(nv+1) + (\tfrac{n}{nv+1}) \bar z_n^2 \bigr\}. \]
So it remains to integrate the right-hand side above with respect to $s$ over the interval $[0,1]$.  Using the identity 
\[ \tfrac{d}{ds} (s - s \log s) = \log(s^{-1}), \]
and the fundamental theorem of calculus, it follows that 
\begin{align*}
\uPi_{z^n}^{\eval} \, \ell_{\hat a(z^n)} & = \int_0^1 n^{-1} \bigl\{ 2\log(s^{-1}) + \log(nv+1) + (\tfrac{n}{nv+1}) \bar z_n^2 \bigr\} \, ds \\
& = n^{-1}\bigl\{ 2 + \log(nv + 1) + (\tfrac{n}{nv+1}) \bar z_n^2 \bigr\}, 
\end{align*}
as was to be shown.

\section{More details about regularizers}
\label{A:other}

Towards a more concrete understanding of what it takes to achieve the condition in Definition~\ref{def:regularizer} in the main text, consider the case of finite $\TT$.  Then Your prior credal set $\credal$ can be interpreted as a collection of probability mass functions, and 
\[ \uprior\rho = \sup_{\prior \in \credal} \underbrace{\sum_{\theta \in \TT} \rho(\theta) \, \prior(\{\theta\})}_{\prior\rho := \E^{\Theta \sim \prior}\{ \rho(\Theta)\}} \leq \sum_{\theta \in \TT} \rho(\theta) \, \sup_{\prior \in \credal} \prior(\{\theta\}). \]
Since the functional $\prior \mapsto \prior\rho$ is linear and the domain $\credal$ is closed and convex, it's well-known (related to the Krein--Milman theorem)
that the supremum is attained on the extreme points of $\credal$.  Then it's clear that, for each (sub-)probability mass function\footnote{A sub-probability mass function $\eta \geq 0$ satisfies $\sum_{\theta \in \TT} \eta(\theta) \leq 1$. In the present case, there's no reason not to take $\eta$ as a genuine probability mass function whose sum is exactly equal to 1.} 
$\eta$ on $\TT$, the following function $\rho_\eta$ is a regularizer:
\[ \rho_\eta(\theta) = \frac{\eta(\theta)}{\sup_{\prior \in \credal} \prior(\{\theta\})}, \quad \theta \in \TT. \]
Note that the form $\rho_\eta$ above mimics the universal inference formulation of an e-process in \citep[e.g.,][]{wasserman.universal}, albeit in a different context.  Two special cases deserve mention.  First, if $\credal$ is a singleton, then all the admissible regularizers would be of the form $\rho_\eta$ above, for some probability mass function $\eta$.  Second, if $\credal$ is vacuous, i.e., if it contains all the probability mass functions on $\TT$, then the denominator is constant equal to 1 and, since $\eta$ can't exceed 1, neither can $\rho_\eta$, hence, it's trivial.  So, clearly, it's a bad idea for You to plug a ``vacuous prior'' into the regularizer construction.  If Your prior information is genuinely vacuous, then just use $\rho \equiv 1$, so that Your regularized e-process matches the original, unregularized e-process.  

In the main paper I focused exclusively on regularizers for the case where Your prior information is encoded as a possibility measure with contour $q$.  I did so for two reasons: (a)~I think the possibilistic formulation makes sense, and (b)~it's relatively concrete to handle compared to other formulations.  But this isn't the only option, so here I want to briefly mention a couple other strategies.  This is absolutely not intended to be an exhaustive list of alternatives, nor is it my intention for this to be a tutorial on how to construct a regularizer in a non-possibilistic setup.  My goal is simply to expose the reader to some other avenues to pursue if the possibilistic version I presented in the main paper isn't fully satisfactory; this also demonstrates my point that the possibilistic formulation is simpler and more concrete. 

Arguably one of the most common imprecise probabilities models are the so-called {\em contamination classes}, {\em gross error models}, or {\em linear--vacuous mixtures}, often found in the literature on robust statistics \citep[e.g.,][]{huber1973.capacity, huber1981, walley1991, walley2002, wasserman1990}.  This corresponds to a choice of centering probability $\prior_\text{cen}$ and a weight $\eps \in (0,1)$.  Then the credal set is given by 
\[ \credal = \bigl\{ (1-\eps) \, \prior_\text{cen} + \eps \, \prior: \, \text{$\prior$ is any probability on $\TT$} \bigr\}. \]
The basic idea is that You think the precise probability $\prior_\text{cen}$ is a pretty good assessment of Your uncertainty about $\Theta$, but You don't fully trust the information that lead to this assessment; so $\eps$ is like the ``chance You were misled,'' and if You were mislead, then literally anything might be true.  It's relatively clear that the corresponding upper probability/prevision, say $\uprior$, has upper expectation of $\rho(\Theta)$ given by 
\[ \uprior\,\rho = (1-\eps) \, \prior_\text{cen}\,\rho + \eps \times \sup_{\theta \in \TT} \rho(\theta). \]
Then the goal is to choose $\rho$ such that the right-hand side above is (less than or) equal to 1.  There's no simple formula for this, like in the case of possibilistic prior information, but the idea is take $\rho$ such that it's relatively small where You ``expect'' $\Theta$ to be, i.e., in the support of $\prior_\text{cen}$, and then somewhat large elsewhere.  If, for example, $\prior_\text{cen}$ has bounded support, so that $\rho$ can be defined independently on that support and elsewhere, then $\rho$ can take values as large as $\eps^{-1}(1-\eps) m$ outside that support, where $m = \prior_\text{cen}\,\rho$. 

A second kind of imprecise probability model is what \citet[][Sec.~2.9.4]{walley1991} calls the {\em constant odds ratio} model.  Similar to that above, this is indexed by a precise probability distribution, which I'll denote again as $\prior_\text{cen}$, and a weight $\tau \in (0,1)$.  Walley explains this model in the context of a risky investment, where $\tau$ represents the rate at which You're taxed on said investment.  Setup and context aside, the lower and upper probability of a hypothesis/event $H$ under this model is 
\[ \lprior(H) = \frac{(1-\tau) \prior_\text{cen}(H)}{1 - \tau \prior_\text{cen}(H)} \quad \text{and} \quad \uprior(H) = \frac{\prior_\text{cen}(H)}{1-\tau \prior_\text{cen}(H)}. \]
The name ``constant odds ratio model'' comes from the fact that the lower and upper odds of $H$ versus $H^c$ are 
\[ \frac{\lprior(H)}{\uprior(H^c)} = \frac{(1-\tau) \prior_\text{cen}(H)}{\prior_\text{cen}(H^c)} \quad \text{and} \quad \frac{\uprior(H)}{\lprior(H^c)} = \frac{\prior_\text{cen}(H)}{(1-\tau)\prior_\text{cen}(H^c)}, \]
respectively, so the rato of lower to upper odds is constant in $H$:
\[ \frac{\lprior(H) / \uprior(H^c)}{\uprior(H) / \lprior(H^c)} = \cdots = (1-\tau)^2. \]
The upper expectation of $\rho(\Theta)$ under this model is not so straightforward as for the contamination model above, but Walley shows that $\uprior\,\rho$ solves the equation $f(x)=0$, where $f(x) = \tau \, \prior_\text{cen}(\rho - x)^+  + (1-\tau) \, ( \prior_\text{cen}\,\rho - x )$, with $w^+ = \max(w, 0)$ the positive part of $w \in \RR$.  For a given $\rho$, one can numerically solve for $x$ as a function of $(\tau, \prior_\text{cen}\,\rho)$, to obtain $\uprior\,\rho$.  But choosing $\rho$ such that this numerical solution is $\leq 1$ requires care. 

Other kinds of imprecise models can be considered, including monotone capacities \citep[e.g.,][]{huber1973.capacity, wasserman.kadane.1990, sundberg.wagner.1992}, belief functions \citep[e.g.,][]{denoeux1999, shafer1976, dempster1967}, and probability-boxes \citep[e.g.,][]{destercke.etal.pbox, ferson.etal.pbox}, and formulas for their upper expectations can be found in, e.g., Chapters~6--7 of \citet{lower.previsions.book}.  



\section{Justification of the product form in Eq.~\eqref{eq:reg.eprocess}}
\label{SS:product}

\subsection{Bayesian-like updating coherence}

Here I'll present the previously-advertised justification for my choice to define the regularized e-process in Equation~(10) as a product of the regularizer and the original e-process.  Actually, I'll present two such justifications, the first of which is based on an analogy to the updating coherence property familiar in Bayesian inference, i.e., ``today's posterior is tomorrow's prior.''  To see this, let $z^n \equiv z^{1:n}$ be the data, processed as $\eval_\theta(z^{1:n})$, and $\rho(\theta)$ the regularizer.  Combining these according to Equation~\eqref{eq:reg.eprocess} gives the regularized e-process $\eval^\text{reg}(z^n, \theta)$.  Now, suppose that more data $z^{n:(n+m)}$ becomes available.  Since e-processes are typically combined via multiplication, I can write 
\[ \eval^\text{reg}(z^{1:(n+m)}, \theta) = \eval_\theta(z^{1:(n+m)}) \times \rho(\theta) = \eval_\theta(z^{n:(n+m)}) \times \underbrace{\eval_\theta(z^{1:n}) \times \rho(\theta)}_{\eval^\text{reg}(z^{1:n}, \theta)}. \]
That is, the regularized e-process based on $z^{1:n}$ becomes the updated, old-data-dependent version of the regularizer that's combined with a new-data-dependent e-process according to the rule in  Equation~\eqref{eq:reg.eprocess} of the main paper. 

\subsection{Formal dominance}

The second justification is based on a more formal demonstration that the product form dominates other strategies.  Frankly, this justification isn't much more compelling than the fact that multiplying the two ingredients is clearly the most natural way to merge them.  There are, however, some other reasonable options, e.g., averaging, so the result below adds some valuable insight.  
The reader may have ideas on how to strengthen this result in one way or another.  

Throughout this section, to simplify notation, etc., I'll assume that $f$ is the identity function, so that the quantity of interest $\Theta$ corresponds exactly with $\Omega$.  This helps because it allows me to drop $f$ (and $\Omega$) from the notation, to drop the set of pullback measures in the theoretical formulation, and to express the model as ``$\prob_\theta$,'' directly in terms of $\theta$.  I'm also going to drop the explicit mention of a generic stopping time, and write ``$Z$'' for the observable data.  Since the main result of this section is conceptual in nature, these simplifications don't affect the take-away message.  

Following \citet{vovk.wang.2021}, define a function $(r,e) \mapsto m(r,e)$ to be a {\em re-merging function}---pronounced ``R-E-merging'' because it merges a {\bf r}egularizer and an {\bf e}-process---if, for  given prior information $\uprior$ about $\Theta=f(\Omega)$ and a corresponding regularizer $\rho$, the merged variable $m(\rho, \eval)$ satisfies two key properties:
\begin{itemize}
\item It treats the data as sovereign in the sense that 
\begin{equation}
\label{eq:sovereign}
e \geq 1 \implies m(r, e) \geq r, \quad \text{for all $r \in (0,\infty)$}, 
\end{equation}
which, in words, means that if the data-dependent component, $\eval_\theta(\cdot)$, offers evidence that doesn't favor hypothesis ``$\Theta=\theta$,'' then the merged e-process will show less support for that hypothesis than the prior information alone did.
\vspace{-2mm}
\item It's a regularized e-process in the sense that the property advertised in Equation~\eqref{eq:reg.ville0} holds for any input $\eval$, i.e., 
\begin{equation}
\label{eq:merging}
\uuprob\bigl[ m\{ \rho(\Theta), \eval_{\Theta}(Z) \} \bigr] \leq 1 \quad \text{for all e-processes $\eval$}. 
\end{equation}
\end{itemize} 
Recall that $\uuprob$ is determined by the model and by the prior information, so $\uuprob$ is fixed by the context of the problem---all that's free to vary in these considerations is the input e-process $\eval$ and the merger function $m$. The class of re-merging functions is non-empty, since the product mapping is an re-merger.  My claim is that, in a sense to be described below, the product merger in Equation~\eqref{eq:reg.eprocess} is ``best'' among all the re-merging functions.

Continuing to follow \citet{vovk.wang.2021}, I'll say that an re-merging function $m$ weakly dominates another re-merging function $m'$ if 
\begin{equation}
\label{eq:dominance}
(r,e) \in [1, \infty) \times [1, \infty) \implies m(r,e) \geq m'(r,e). 
\end{equation}
The idea is that, in the case where neither the data nor the prior show signs of compatibility with a given value $\theta$, then merging based on $m$ is more aggressive, i.e., shows no less evidence against $\theta$, than merging based on $m'$.  The complement to an e-process's anytime validity property is its efficiency, and efficiency requires that the e-process take large values when evidence is incompatible with a hypothesis in question; so, an re-merging function that returns larger regularized e-process values is preferred.  

The present case differs in many ways from that in \citet{vovk.wang.2021}, mainly due to the presence of the prior information, so I'll need some additional control.  Towards this, I'll say that an re-merging function $m$ {\em $\uuprob$-strictly weakly dominates} another re-merging function $m'$ if $m$ weakly dominates $m'$ in the sense of \eqref{eq:dominance} above, and if
\begin{align}
\uuprob \bigl[ m\{ \rho(\Theta), \eval_\Theta(Z) \} > m'\{ \rho(\Theta), \eval_\Theta(Z) \}, \; & \notag \\
\rho(\Theta) \geq 1, \, \eval_\Theta(Z) \geq 1 & \bigr] > 0, \quad \text{for all e-processes $\eval$}. \label{eq:uuprob.dominance}
\end{align}
Vovk and Wang's ``weak dominance'' in \eqref{eq:dominance} allows $m$ and $m'$ to be the same, so roughly all that \eqref{eq:uuprob.dominance} adds is that there exists a joint distribution for $(Z,\Theta)$---corresponding to a prior $\prior$ in Your credal set $\credal$---with respect to which the ``strict inequality'' event on the right-hand side of \eqref{eq:dominance} has positive probability.  Therefore, $\uuprob$-strict weak dominance rules out the possibility that $m$ and $m'$ differ only in an insignificant way relative to $\uuprob$.  In the special case where the prior information is vacuous, like in Vovk and Wang, strict weak dominance holds if $m$ weakly dominates $m'$ and if, for each e-process $\eval$, there exists a $\theta=\theta(\eval)$ such that $\rho(\theta) \geq 1$ and
\[ \prob_\theta\bigl[ m\{ \rho(\theta), \eval_\theta(Z) \} > m'\{ \rho(\theta), \eval_\theta(Z) \}, \, \eval_\theta(Z) \geq 1 \bigr] > 0, \]
i.e., if roughly strict inequality holds with positive model probability for some $\theta$.  

In the case of vacuous prior information, I'll require $\rho \equiv 1$, as was suggested in the main paper.  When the prior information is non-vacuous, the result below is restricted to {\em non-trivial} regularizers, which was loosely defined earlier as a regularizer that's not upper bounded by 1.  Here, however, I need to be more specific about what non-trivial means: specifically, $\rho$ is non-trivial (relative to the partial prior information) if $\uprior\{ \rho(\Theta) > 1 \} > 0$.  In words, non-triviality means that there exists a probability $\prior$ such that $\rho(\Theta)$ isn't $\prior$-almost surely upper bounded by 1. Finally, I'll also say that a regularizer is {\em admissible} if $\uprior\rho=1$, that is, if it can't be made larger in any substantive way without violating the upper bound presented in Definition~\ref{def:regularizer}. 

The following is similar to Proposition~4.2 in \citet{vovk.wang.2021} on what they refer to as {\em ie-merging} functions for merging independent e-values.  I show that the product mapping isn't strictly weakly dominated by any other re-merging functions.  

\begin{prop}
\label{prop:dominates}
Given $\uprior$, fix a non-trivial, admissible regularizer $\rho$.  Then the product rule in Equation~\eqref{eq:reg.eprocess} isn't $\uuprob$-strictly weakly dominated by any other re-merging function. 
\end{prop}

\begin{proof}
The proof is by contradiction; that is, I'll assume that the product rule defined in Equation~\eqref{eq:reg.eprocess} is $\uuprob$-strictly weakly dominated in the sense above and show that this leads to a contradiction.  Let $m$ denote this assumed-to-exist dominant re-merging function. 

The assumed non-triviality of the regularizer $\rho$ implies existence of points $\theta$ such that $\rho(\theta) > 1$ and such that $\rho(\theta) \leq 1$, and neither of these sets have $\uprior$-probability 0.  Some of the $\theta$'s in the first set could have $\rho(\theta) = \infty$, but, those must have  $\uprior$-probability 0 for, otherwise, the property that $\uprior\,\rho \leq 1$ would be violated; recall that the assumed admissibility of $\rho$ means that $\uprior\,\rho=1$.  

From the assumed strict weak dominance of $m$, and from \eqref{eq:sovereign}, I can deduce the following bound for generic inputs $(r,e)$:
\begin{align*}
m(r,e) & = m(r,e) \, \bigl( 1_{r < 1, e < 1} + 1_{r \geq 1, e < 1} + 1_{r < 1, e \geq 1} + 1_{r \geq 1, e \geq 1} \bigr) \\
& \geq  m(r,e) \, 1_{r < 1, e < 1} + m(r,e) \, 1_{r \geq 1, e < 1} + r \, 1_{r < 1, e \geq 1} + re \, 1_{r \geq 1, e \geq 1},
\end{align*}
where the ``$r$'' factor in the third term is by \eqref{eq:sovereign} and the ``$re$'' factor in the fourth term is by the assumed weak dominance.  The first two terms depend on details of the particular choice of $m$ on the respective ranges of $(r,e)$, details that can't be controlled with only the information provided.  I can apply the trivial non-negativity bound, however, and, from the above display, conclude that 
\[ m(r,e) \geq r \, 1_{r < 1, e \geq 1} + re \, 1_{r \geq 1, e \geq 1}. \]
(In fact, with the input e-process to be constructed next, those two terms I lower-bounded by 0 would typically be equal to 0, so the above inequality isn't loose.)  The not-strict equality in the above display is pointwise in $(r,e)$, i.e., I can't rule out equality above for any given pair $(r,e)$.  But a pointwise lower bound isn't the goal---I'm aiming for a lower bound in upper expectation.  For this latter goal, I'll apply \eqref{eq:uuprob.dominance} to flip the not-strict inequality ``$\geq$'' in the above display a strict inequality; more on this below.  

Towards establishing a contradiction, I only need to produce one example of an input e-process such that the conclusion is problematic, and I'll do this with an incredibly simple e-process.  Specifically, I define the input e-process $\eval^\star$ as follows:
\begin{itemize}
\item if $\theta$ is such that $\rho(\theta) < 1$, then $\eval_\theta^\star(Z) \equiv 1$, and 
\vspace{-2mm}
\item if $\theta$ is such that $\rho(\theta) \geq 1$, then  
\[ \eval_\theta^\star(Z) = \begin{cases} 2 & \text{if $Z \in \mathcal{E}_\theta$} \\ 0 & \text{otherwise}, \end{cases} \]
where $\mathcal{E}_\theta$ is an event with $\prob_\theta(\mathcal{E}_\theta) = \frac12$.
\end{itemize} 
It's easy to check that $\E_\theta\{ \eval_\theta^\star(Z) \} = 1$ for all $\theta$, so $\eval^\star$ is a genuine e-process.  What's important about this particular e-process, as it pertains to the present proof, is that 
\[ 1_{\rho(\theta) < 1, \eval_\theta^\star \geq 1} = 1_{\rho(\theta) < 1} \quad \text{and} \quad \eval_\theta^\star \, 1_{\rho(\theta) \geq 1, \eval_\theta^\star \geq 1} = \eval_\theta^\star \, 1_{\rho(\theta) \geq 1}. \]
Therefore, from the pointwise analysis above, 
\begin{align*}
m\{ \rho(\Theta), \eval_\Theta^\star(Z) \} & \overset{\text{\tiny (+)}}{\geq} \rho(\Theta) \, 1_{\rho(\Theta) < 1, \eval_\Theta^\star(Z) \geq 1} + \eval_\Theta^\star(Z) \, \rho(\Theta) \, 1_{\rho(\Theta) \geq 1, \eval_\Theta^\star(Z) \geq 1} \\
& = \rho(\Theta) \, 1_{\rho(\Theta) < 1} + \eval_\Theta^\star(Z) \, \rho(\Theta) \, 1_{\rho(\Theta) \geq 1}. 
\end{align*}
The (+) symbol above is to remind the reader that, while ``$\geq$'' holds pointwise, there's actually more that can be said.  That is, in addition to ``$\geq$'' for every $(z,\theta)$ pair, there exists a joint distribution for $(Z,\Theta)$, compatible with the available prior information, such that strict inequality holds with positive probability.  This implies that the inequality ``$\geq$'' highlighted with (+) is {\em strict inequality} ``$>$'' in expectation with respect to the aforementioned joint distribution.  Then 
\begin{align}
\uuprob\bigl[ m\{ \rho(\Theta), \eval_\Theta^\star(Z) \} \bigr] & > \uuprob\bigl\{ \rho(\Theta) \, 1_{\rho(\Theta) < 1} + \eval_\Theta^\star(Z) \, \rho(\Theta) \, 1_{\rho(\Theta) \geq 1} \bigr\} \notag \\
& = \sup_{\prior \in \credal} \E^{(Z,\Theta) \sim \prob_\bullet \otimes \prior} \bigl\{ \rho(\Theta) \, 1_{\rho(\Theta) < 1} + \eval_\Theta^\star(Z) \, \rho(\Theta) \, 1_{\rho(\Theta) \geq 1} \bigr\} \notag \\
& = \uprior\bigl\{ \rho(\Theta) \, 1_{\rho(\Theta) < 1} + \rho(\Theta) \, 1_{\rho(\Theta) \geq 1} \bigr\} \label{eq:dominance.a} \\
& = \uprior\,\rho \notag \\
& = 1, \label{eq:dominance.b}
\end{align}
where \eqref{eq:dominance.a} holds because $\eval^\star$ is an e-process relative to the model, i.e., $\E_\theta(\eval_\theta^\star) = 1$ for all $\theta$, and \eqref{eq:dominance.b} holds by the assumed admissibility of the regularizer $\rho$.  Therefore, the defining property \eqref{eq:merging} of an re-merging fails for the chosen $m$; that is, I constructed an e-process $(\eval^\star)$ such that $\uuprob[ m\{ \rho(\Theta), \eval_\Theta^\star(Z) \} ] > 1$.  Since $m$ was assumed to be re-merging, this creates the desired contradiction.  So, I conclude that, as claimed, there's no re-merging function $m$ that $\uuprob$-strictly weakly dominates the product rule in Equation~\eqref{eq:reg.eprocess}.  
\end{proof}

\section{More efficiency gains}
\label{S:emp.efficiency}

To follow up on the illustrations given in the main text, here I'll consider two different types of prior information and associated prior possibility contours.  As before, these both will be indexed by a parameter $K$, although the meaning of $K$ will be different in each case.  Consequently, the corresponding regularized e-processes will not be comparable between the two types of prior information here, but they will be comparable as $K$ varies within the types.  Here, $K \in \{0.1, 0.2, 0.4, 0.8\}$ and the two types of prior information I'll consider are below; first an explanation in words and then a mathematical description. 
\begin{enumerate}
\item ``You expect that $|\Theta| \leq K$.'' This is pretty clear in words but, mathematically, this corresponds to a credal set $\credal$ that contains exactly those distributions $\prior$ such that $\E^{\Theta \sim \prior}|\Theta| \leq K$.  Of course, this includes certain Gaussian, uniform, and even some heavier-tailed priors. As shown in \citet{dubois.etal.2004} and elsewhere, this prior information can be described mathematically via the possibility contour 
\[ q(\theta) = 1 \wedge K|\theta|^{-1}, \quad \theta \in \TT. \]
\item ``You're at most $100K/5\%$ sure that $|\Theta| > 2K$.'' This means that You can't rule out the possibility of $|\Theta| \leq 2K$, but that You judge the probability of the event $|\Theta| > 2K$ to be upper-bounded by $K/5$.  Mathematically, this can be described easily by the possibility contour 
\[ q(\theta) = \tfrac{K}{5} \, 1(|\theta| > 2K) + 1( |\theta| \leq 2K ), \quad \theta \in \TT. \]
I'm using the $K$-dependent weights only so that the different lines in the plots shown in Figure~\ref{fig:prob} don't overlap on a large part of the range of $\theta$ values. 
\end{enumerate} 

These two kinds of prior information are rather weak, only based on some very basic judgments about a first moment and the probability of a single event.  Figures~\ref{fig:mean} and \ref{fig:prob} plot the regularized and unregularized log-transformed e-process as a function of $\theta$ for the three types of priors, respectively, and for three different values of the observed sample mean $\bar z$ as in the main text, based on a sample of size $n=5$.
The black line corresponds to the unregularized e-process, and the four colored lines correspond to the different values of $K$; the dashed horizontal line corresponds to $-\log0.05 \approx 3$, which is the cutoff that determines the (regularized) e-process's 95\% confidence interval.  Not surprisingly, given that prior Types~1--2 are rather weak, the effect of regularization is difficult to detect.  There is a small efficiency gain in Figures~\ref{fig:mean}--\ref{fig:prob}, more so in the latter, since the prior does more than strictly constrain $\Theta$ to an interval around the origin.  

\begin{figure}[t]
\begin{center}
\subfigure[$\bar z=0.25$]{\scalebox{0.41}{\includegraphics{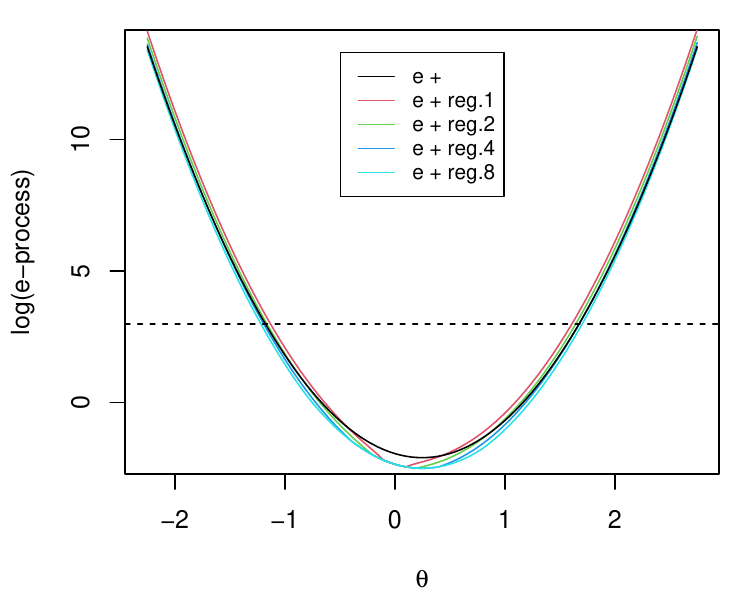}}}
\subfigure[$\bar z=0.5$]{\scalebox{0.41}{\includegraphics{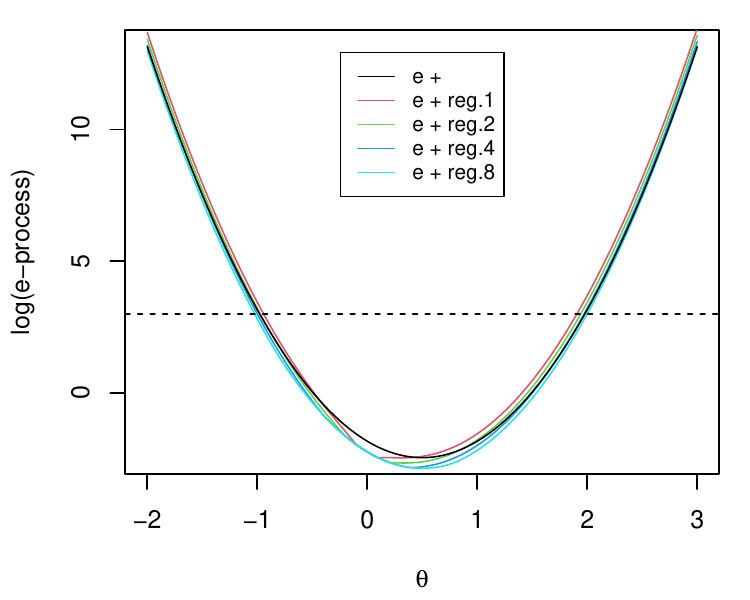}}}
\subfigure[$\bar z=1$]{\scalebox{0.41}{\includegraphics{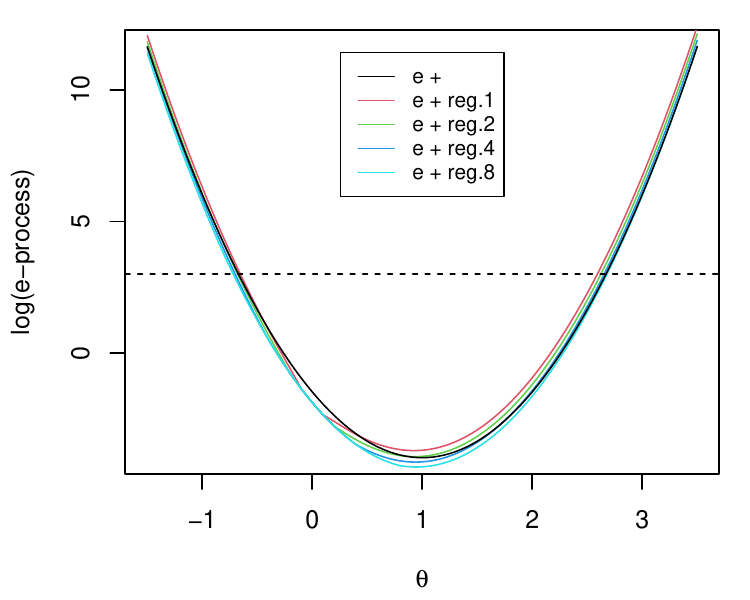}}}
\end{center}
\caption{Plot of $\theta \mapsto \eval^\text{reg}(z^n, \theta)$ for three different data sets $z^n$ based on prior Type~1.}
\label{fig:mean}
\end{figure}

\begin{figure}[t]
\begin{center}
\subfigure[$\bar z=0.25$]{\scalebox{0.41}{\includegraphics{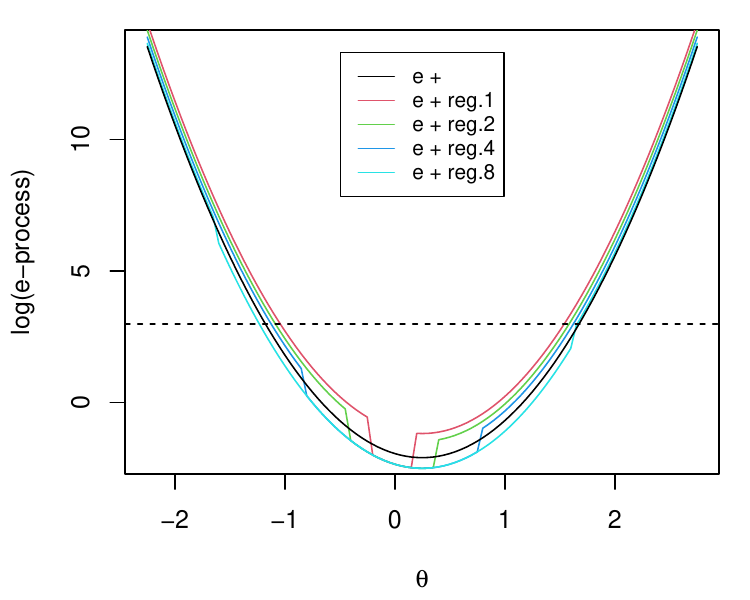}}}
\subfigure[$\bar z=0.5$]{\scalebox{0.41}{\includegraphics{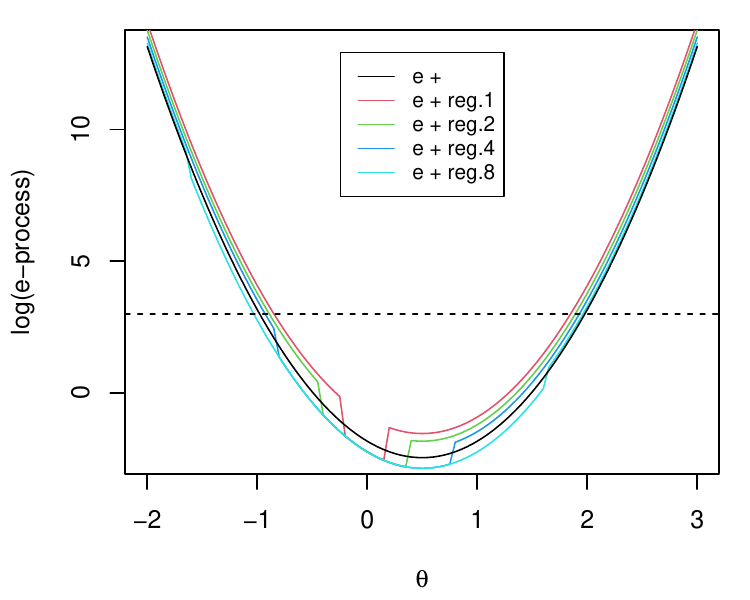}}}
\subfigure[$\bar z=1$]{\scalebox{0.41}{\includegraphics{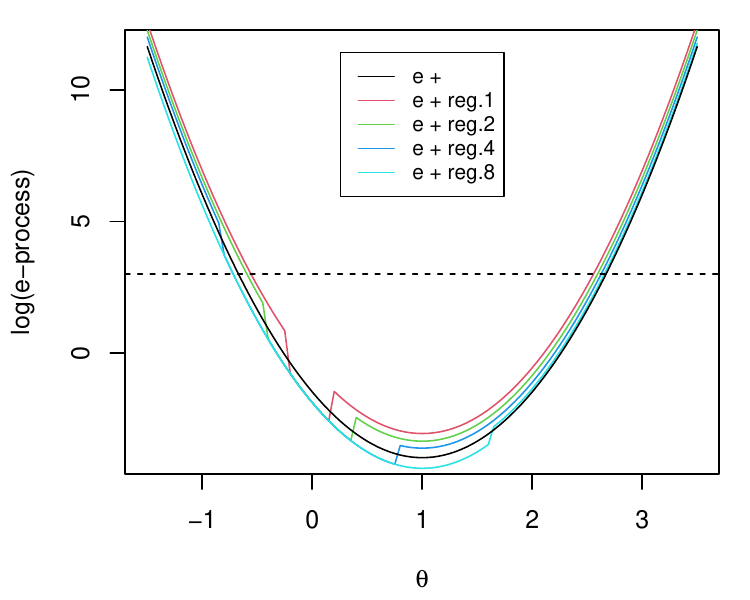}}}
\end{center}
\caption{Plot of $\theta \mapsto \eval^\text{reg}(z^n, \theta)$ for three different data sets $z^n$ based on prior Type~2.}
\label{fig:prob}
\end{figure}

\section{Anytime validity implies no-sure-loss}
\label{S:euq.coherence}

When data $Z^N=z^n$ is fixed, the e-possibilistic IM typically defines an  imprecise probability model where $\uPi_{z^n}^{\eval \times \rho}$ has the mathematical properties of a possibility measure.  Provided that the function $\theta \mapsto \pi_{z^n}^{\eval \times \rho}(\theta)$ isn't bounded away from 1, which is typically the case, the function $H \mapsto \uPi_{z^n}^{\eval \times \rho}(H)$ determined by maximizing a possibility contour function over the set $H \subseteq \TT$ as in \eqref{eq:upper.p.poss} is a genuine possibility measure.  This implies that its credal set is non-empty and, in turn, that the IM output is {\em coherent} in the sense of \citet[][Sec.~2.5]{walley1991}.  For the present purposes, it's enough to understand coherence as a stronger version of {\em no-sure-loss}, which can be described as follows. Suppose beliefs concerning the uncertain value $\Theta$ are assessed via the buying/selling prices they consider acceptable for certain gambles about $\Theta$.  Let $z^n$ be one of the aforementioned typical data sets and let Your buying/selling prices be determined by the IM output $(\lPi_{z^n}^{\eval \times \rho}, \uPi_{z^n}^{\eval \times \rho})$ according to the interpretation in \eqref{eq:prices}.  Then coherence of Your IM implies that
\[ \sup_{\theta \in \TT} \sum_{k=1}^K \{ \uPi_{z^n}^{\eval \times \rho}(H_k) - 1(\theta \in H_k)\} \geq 0 \quad \text{for all $K$ and $(H_1,\ldots,H_K)$ combos}. \] 
To get the intuition, suppose the above condition fails.  Then there exists a combination $K$ and $(H_1,\ldots,H_K)$, and a sufficiently small $\delta > 0$, such that 
\[ \sup_{\theta \in \TT} \sum_{k=1}^K \bigl[ \{ \uPi_{z^n}^{\eval \times \rho}(H_k) + \delta\} - 1(\theta \in H_k) \bigr] < 0. \]
Since $\uPi_{z^n}^{\eval \times \rho}(H_k)$ is, by definition, Your infimum selling price for the gamble $1(\Theta \in H_k)$, the transactions where You accept payment of $\uPi_{z^n}^{\eval \times \rho}(H_k) + \delta$ dollars for \$$1(\Theta \in H_k)$ for each $k$ are all acceptable to You.  But then the above display reveals a troubling result: somehow, by only making transactions that are acceptable {\em a priori}, You end up with negative total earnings regardless of what value the uncertain $\Theta$ takes on.  This indicates a severe shortcoming in Your pricing scheme; fortunately, the e-possibilistic IM is typically free of this internal inconsistency.  

I said ``typically'' several times in the above paragraph, and the explanation here is exactly the same as in Remark~\ref{re:typically} above.  The IM output $\uPi_{z^n}^{\eval \times \rho}$ would fail to be a possibility measure if and only if $\pi_{z^n}^{\eval \times \rho}$ was bounded away from 1 on $\TT$; I've been calling that function a ``possibility contour'' but that's only legitimate if $\sup_{\theta \in \TT} \pi_{z^n}^{\eval \times \rho}(\theta) = 1$.  For $\pi_{z^n}^{\eval \times \rho}(\theta)$ to be strictly less than 1, and hence the IM output determines an incoherent imprecise probability, would require that $\theta \mapsto \eval^\text{reg}(z^n, \theta)$ also be bounded strictly greater than 1.  But the regularized Ville's inequality implies that $\uuprob(\eval^\text{reg}) \leq 1$, i.e., $\eval^\text{reg}(Z^N,\Theta)$ ``tends'' to be less than 1, so a data set $z^n$ could indeed be called atypical if it were such that $\eval^\text{reg}(z^n, \theta)$ were strictly greater than 1 for all $\theta$.  

In addition the fixed-data behavioral considerations, it's natural to interpret $(\credal, z^n) \mapsto (\lPi_{z^n}^{\eval \times \rho}, \uPi_{z^n}^{\eval \times \rho})$ as a rule by which ``prior'' information is {\em updated} in light of data ($z^n$) to a ``posterior'' quantification of uncertainty.  More familiar notions of imprecise-probabilistic updating include generalized Bayes rule \citep{walley1991, miranda.cooman.chapter} and Dempster's rule \citep[e.g.,][]{shafer1976, cuzzolin.book}.
With this ``updating rule'' interpretation comes further questions about the IM's ability to protect You from sure loss, etc.  What's different here is that there's a temporal component: can I force You into transactions such that, no matter what data is observed, You lose money?  If so, then there's a serious issue with Your assessments. Mathematically, this {\em sure loss} property---see, e.g., \citet[][Sec.~2.4.1]{walley1991} and \citet[][Def.~3.3]{gong.meng.update}---corresponds to existence of a hypothesis $H \subset \TT$ such that 
\begin{equation}
\label{eq:sure.loss}
\sup_{z^n} \uPi_{z^n}^{\eval \times \rho}(H) < \lprior(H) \quad \text{or} \quad  \inf_{z^n} \lPi_{z^n}^{\eval \times \rho}(H) > \uprior(H). 
\end{equation}
For intuition, consider the first of the above two inequalities.  If this inequality holds, then, for any pair of positive numbers $(\eps, \delta)$, You'd be willing to buy the gamble \$$1(\Theta \in H)$ from me for $\$\{\lprior(H) - \eps\}$ and then sell the same gamble back to me, after observing $z^n$, for \$$\{\sup_{z^n} \uPi_{z^n}^{\eval \times \rho}(H) + \delta\}$.  No matter whether $\Theta \in H$ or $\Theta \not\in H$, the payoff You receive from this sequence of transactions is 
\[ \Bigl\{ \sup_{z^n} \uPi_{z^n}^{\eval \times \rho}(H) + \delta \Bigr\} - \{\lprior(H) - \eps\} = \Bigl\{ \underbrace{\sup_{z^n} \uPi_{z^n}^{\eval \times \rho}(H) - \lprior(H)}_{\text{$<0$, by \eqref{eq:sure.loss}}} \Bigr\} + (\eps + \delta). \]
Since both individual transactions are acceptable to You, and there exists pairs $(\eps, \delta)$ such that You net payoff is strictly negative, it follows that You can be made a sure loser.  

Fortunately, the possibilistic IM provably avoids an even less severe internal inconsistency, which I call {\em one-sided contraction}:
\begin{equation}
\label{eq:contraction}
\sup_{z^n} \uPi_{z^n}^{\eval \times \rho}(H) < \uprior(H) \quad \text{or} \quad \inf_{z^n} \lPi_{z^n}^{\eval \times \rho}(H) > \lprior(H). 
\end{equation}
One-sided contraction is less concerning than sure-loss, but still problematic.  To see this, suppose that the first of the two inequalities in \eqref{eq:contraction} holds for a given $H$.  If I want to buy \$$1(\Theta \in H)$ from You {\em a priori}, then it's apparent that I can wait until data $z^n$ is revealed and purchase the gamble for a lower price.  This doesn't imply that You lose money, only that You're systematically giving up opportunities to earn more money.  This can't happen in the familiar, precise Bayesian case: since the expected value of a posterior probability is the prior probability, it's impossible for each posterior/conditional probability to be less than the prior/marginal probability.  

The stronger notion of (two-sided) {\em contraction} corresponds to replacing the ``or'' in the above display with ``and,'' and the dual notion of {\em dilation} corresponds to flipping both inequalities and replacing ``or'' with ``and.'' Both contraction and dilation are problematic in their own respects, but contraction is generally more serious.  Of course, if an e-possibilistic IM avoids one-sided contraction, as the next theorem establishes, then it necessarily avoids both sure-loss and (two-sided) contraction.  

\begin{thm}
\label{thm:no.sure.loss}
The e-possibilistic IM avoids one-sided contraction, i.e., there are no $H$ such that \eqref{eq:contraction} holds for the updating rule $(\credal, z^n) \mapsto (\lPi_{z^n}^{\eval \times \rho}, \uPi_{z^n}^{\eval \times \rho})$. 
\end{thm}

\begin{proof}
Fix any hypothesis $H \subseteq \TT$.  I'll focus on proving that the first inequality in \eqref{eq:contraction} doesn't hold, i.e., that 
\begin{equation}
\label{eq:no.contraction}
\sup_{n,z^n} \uPi_{z^n}^{\eval \times \rho}(H) \geq \uprior(H). 
\end{equation}
The version involving lower probabilities is proved similarly.  To start, note that 
\[ \inf_{n, z^n} \eval_\theta(z^n) \leq \sup_N \sup_{\omega: f(\omega)=\theta} \E_\omega \{ \eval_\theta(Z^N) \} \leq 1, \]
where the first inequality is because an average is never smaller than the minimum, and the second inequality by Ville.  That is, for any $\theta$, there exists $(n,z^n)$ such that $\eval_\theta(z^n)$ is no more than 1.  Now write out the left-hand of \eqref{eq:no.contraction} as follows:
\[ \sup_{n,z^n} \uPi_{z^n}^{\eval \times \rho}(H) = \sup_{n,z^n} \sup_{\theta \in H} \pi_{z^n}^{\eval \times \rho}(\theta) = 1 \wedge \Bigl\{ \inf_{n,z^n} \inf_{\theta \in H} \rho(\theta) \, \eval_\theta(z^n) \Bigr\}^{-1}. \]
It follows from the ``first observation'' above that 
\[ \sup_{n,z^n} \uPi_{z^n}^{\eval \times \rho}(H) \geq 1 \wedge \Bigl\{ \inf_{\theta \in H} \rho(\theta) \Bigr\}^{-1}. \]
Since $\rho(\theta) \geq \{ \inf_{\theta \in H} \rho(\theta) \} \times 1(\theta \in H)$, it's easy to see that 
\[ \Bigl\{ \inf_{\theta \in H} \rho(\theta) \Bigr\} \times \uprior(H) \leq \uprior\,\rho \leq 1, \]
where the first inequality is by monotonicity of the upper expectation and the second by definition of the regularizer $\rho$.  Plugging this bound into that above gives 
\[ \sup_{n,z^n} \uPi_{z^n}^{\eval \times \rho}(H) \geq 1 \wedge \uprior(H) = \uprior(H), \]
which completes the proof of the theorem. 
\end{proof}

An even stronger internal consistency property, a similarly-temporal version of {\em coherence}, might be desired.  Suffice it to say that the e-possibilistic IM generally satisfies only one of the two necessary and sufficient conditions \citet[][Sec.~6.5.2]{walley1991} for coherence.  There are cases where coherence can be achieved, but it can't be achieved in general; see \citet[][Sec.~3.3]{martin.partial} for more discussion on this point. 

\section{Proofs from the main paper}

\subsection{Proof of Proposition~\ref{prop:prior.reg}}

By definition of the $\uprior$-upper expectation, 
\[ \uprior \, \rho = \int_0^1 \Bigl\{ \sup_{\theta: q(\theta) > s} \rho(\theta) \Bigr\} \, ds = \int_0^1 \Bigl\{ \sup_{\theta: q(\theta) > s} \frac{1}{\gamma \circ q(\theta)} \Bigr\} \, ds. \]
Since $\gamma$ is non-decreasing, 
\[ q(\theta) > s \implies \frac{1}{\gamma \circ q(\theta)} \leq \frac{1}{\gamma(s)}, \]
and from here the claim follows immediately from Equation~\eqref{eq:calibrate}.  

\subsection{Proof of Proposition~\ref{prop:pullback}}

Take any fixed $\prior \in \credal$.  Non-emptiness of $\pbcred_\prior$ is a consequence of the classical results mentioned above.  To prove convexity, take two elements $\pb_1$ and $\pb_2$ in $\pbcred_\prior$ and a constant $\tau \in [0,1]$, and then define the mixture $\pb^\text{mix} = (1-\tau) \pb_1 + \tau \pb_2$; the goal is to show that $\pb^\text{mix} \in \pbcred_\prior$.  If $\Omega \sim \pb^\text{mix}$, then for any event $H$ in $\TT = f(\OO)$, 
\begin{align*}
\pb^\text{mix}\{ f(\Omega) \in H \} & = (1-\tau) \, \pb_1\{ f(\Omega) \in H\} + \tau \, \pb_2\{ f(\Omega) \in H \} \\
& = (1-\tau) \, \prior(H) + \tau \, \prior(H) \\
& = \prior(H).
\end{align*}
This implies that $\pb^\text{mix} \in \pbcred_\prior$ if $\pb_1$ and $\pb_2$ are and, therefore, that $\pbcred_\prior$ is convex.  Finally, to prove that $\pbcred_\prior$ is closed with respect to the weak topology, consider a sequence $(\pb_t: t \geq 1)$ in $\pbcred_\prior$ with a weak limit $\pb_\infty$; the goal is to show that $\pb_\infty \in \pbcred_\prior$.  By the continuous mapping theorem, if $\Omega_t \sim \pb_t$, then $f(\Omega_t) \to f(\Omega_\infty)$ in distribution as $t \to \infty$.  But the distribution of $f(\Omega_t)$ is $\prior$ for all $t$ and, consequently, the distribution of $f(\Omega_\infty)$ is also $\prior$.  This implies $\pb_\infty$ is contained in $\pbcred_\prior$ and, hence, the latter collection is closed. 

\subsection{Proof of Theorem~\ref{thm:reg.ville}} 

The first claim is just a direct computation using Equation~\eqref{eq:uuprob} when the function $g = \eval^\text{reg}$ factors as $g(\cdot, \omega) = \rho(f(\omega)) \, \eval_{f(\omega)}(\cdot)$:
\begin{align*}
\uuprob \, \eval^\text{reg} & = \sup_{\prior \in \credal} \sup_{\pb \in \pbcred_\prior} \E^{\Omega \sim \pb} \bigl[ \E^{Z \sim \prob_\Omega}\{ \eval^\text{reg}(Z^N, f(\Omega)) \} \bigr] \\
& = \sup_{\prior \in \credal} \sup_{\pb \in \pbcred_\prior} \E^{\Omega \sim \pb} \bigl[ \rho( f(\Omega) ) \,  \underbrace{\E^{Z \sim \prob_\Omega}\{ \eval_{f(\Omega)}(Z^N) \}}_{\text{$\leq 1$, by Equation \eqref{eq:eval.bound}}} \bigr] \\
& \leq \sup_{\prior \in \credal} \sup_{\pb \in \pbcred_\prior} \E^{\Omega \sim \pb} \bigl[ \rho\{ f(\Omega) \} \bigr] \\
& = \sup_{\prior \in \credal} \E^{\Theta \sim \prior} \{ \rho(\Theta) \} \\
& \leq 1,
\end{align*}
where the last inequality follows by definition of the regularizer $\rho$, and the penultimate equality follows by definition of $\pbcred_\prior$: if $\Omega \sim \pb \in \pbcred_\prior$, then the distribution of $f(\Omega)$ is the same as that of $\Theta$ under $\prior$.  The second claim, the regularized Ville's inequality in Equation~\eqref{eq:reg.ville}, follows from the first claim and an application of Markov's inequality inside the upper-probability calculation.  If I write ``$(Z,\Omega) \sim \prob_\bullet \otimes \pb$'' to represent the joint distribution of $(Z,\Omega)$ under the model where $\Omega \sim \pb$ and $(Z \mid \Omega=\omega) \sim \prob_\omega$, then:
\begin{align*}
\uuprob \bigl[ \eval^\text{\rm reg}\{Z^N, f(\Omega)\} > \alpha^{-1} \bigr] & = \sup_{\prior \in \credal} \sup_{\pb \in \pbcred_\prior} \prob^{(Z,\Omega) \sim \prob_\bullet \otimes \pb}\bigl[ \eval^\text{reg}\{Z^N, f(\Omega) \} > \alpha^{-1} \bigr] \\
& \leq \sup_{\prior \in \credal} \sup_{\pb \in \pbcred_\prior} \alpha \, \E^{\Omega \sim \pb} \bigl[ \E^{Z \sim \prob_\Omega}\{ \eval^\text{reg}(Z^N, f(\Omega)) \} \bigr] \\
& = \alpha \, \uuprob \, \eval^\text{reg} \\
& \leq \alpha,
\end{align*}
where the first ``$\leq$'' above is by the usual Markov's inequality and the last line is by the claim in Equation~\eqref{eq:reg.ville0} proved above.  

\subsection{Proof of Theorem~\ref{thm:valid.vac}}

Since the IM with contour $\pi^\eval$ dominates that with contour $\pi$ defined earlier, the anytime validity of the latter implies that of the former.  For concreteness, however, I'll give a direct proof of anytime validity of $\uPi^\eval$.  

For any data set $z^n$, the IM's possibility measure output $H \mapsto \uPi_{z^n}(H)$ is monotone.  Therefore, if $\omega$ is such that $f(\omega) \in H$, then 
\[ \uPi_{z^n}^\eval(H) \geq \uPi_{z^n}^\eval(\{ f(\omega) \}) = \pi_{z^n}^\eval(f(\omega)), \]
and, consequently, for any $\alpha \in [0,1]$, 
\[ \uPi_{z^n}(H) \leq \alpha \implies \pi_{z^n}^\eval(f(\omega)) \leq \alpha \iff \eval_{f(\omega)}(z^n) \geq \alpha^{-1}, \]
where the right-most property is by definition of the IM's contour function in terms of the e-process's reciprocal.  It follows that 
\[ \sup_{\omega: f(\omega) \in H} \prob_\omega\{ \uPi_{Z^N}^\eval(H) \leq \alpha \} \leq \sup_{\omega: f(\omega) \in H} \prob_\omega\{ \eval_{f(\omega)}(Z^N) \geq \alpha^{-1} \} \leq \alpha, \]
with the last inequality due to Equation~\eqref{eq:ville}, thus completing the proof.

\subsection{Proof of Corollary~\ref{cor:uniform.vac}}

The proof is based on the following observation:
\begin{equation}
\label{eq:equivalence} 
\uPi_{Z^N}^\eval(H) \leq \alpha \text{ for some $H$ with $H \ni f(\omega)$} \iff \pi_{Z^N}^\eval(f(\omega)) \leq \alpha. 
\end{equation}
The ``$\Longleftarrow$'' direction is obvious since $H := \{f(\omega)\}$ is a hypothesis that contains $f(\omega)$.  The ``$\Longrightarrow$'' is similarly obvious by the monotonicity property as stated in Equation~\eqref{eq:monotone}. Since Corollary~1 says the right-most event in the above display has probability no more than $\alpha$, uniformly in $\omega$ and in $N$, the same must be true of the equivalent left-most event.  

\subsection{Proof of Theorem~\ref{thm:imdec}}

Write $\Theta$ instead of $f(\Omega)$ for now.  First, observe that the integrand 
\[ s \mapsto \sup\{\ell_a(\theta): \pi_{Z^N}^{\eval \times \rho}(\theta) \geq s\} \]
in $\uPi_{Z^N}^{\eval \times \rho}(\ell_a)$ is non-decreasing.  Next, I proceed by considering two separate cases.  
\begin{enumerate}
\item Case: $\eval^\text{reg}(Z^N,\Theta) > 1$.  In this case,  $\pi_{Z^N}^{\eval \times \rho}(\Theta) = 1$ so, by the monotonicity property mentioned above, the pair $(Z^N,\Theta)$ is such that 
\[ \sup_{\theta: \pi_{Z^N}^{\eval \times \rho}(\theta) \geq s} \ell_a(\theta) \geq \ell_a(\Theta) \quad \text{for all $s \in [0,1]$}. \]
Of course, if the integrand is bounded below, then the integral, over all of $[0,1]$, is also bounded below by the same value.  Therefore, 
\[ \uPi_{Z^N}^{\eval \times \rho}(\ell_a) \geq \ell_a(\Theta), \]
which implies 
\begin{equation}
\label{eq:risk.bound}
\frac{\ell_a(\Theta)}{\uPi_{Z^N}^{\eval \times \rho}(\ell_a)} \leq 1 < \eval^\text{reg}(Z^N,\Theta). 
\end{equation}
\item Case: $\eval^\text{reg}(Z^N,\Theta) \leq 1$.  In this case, $\pi_{Z^N}^{\eval \times \rho}(\Theta) \leq 1$, so I can lower-bound the upper expected loss by truncating the range of integration as follows:
\begin{align*}
\uPi_{Z^N}^{\eval \times \rho}(\ell_a) & = \int_0^1 \Bigl\{ \sup_{\theta: \pi_{z^n}^{\eval \times \rho}(\theta) \geq s} \ell_a(\theta) \Bigr\} \, ds \\
& \geq \int_0^{\pi_{Z^N}^{\eval \times \rho}(\Theta)} \Bigl\{ \sup_{\theta: \pi_{z^n}^{\eval \times \rho}(\theta) \geq s} \ell_a(\theta) \Bigr\} \, ds \\
& \geq \pi_{Z^N}^{\eval \times \rho}(\Theta) \, \ell_a(\Theta),
\end{align*}
where the last inequality is again by the monotonicity property highlighted above.  On rearranging, and using the fact that $\pi_{Z^N}^{\eval \times \rho}(\Theta) = \eval_\Theta(Z^N) \, \rho(\Theta)$ in this case, I get 
\[ \frac{\ell_a(\Theta)}{\uPi_{Z^N}^{\eval \times \rho}(\ell_a)} \leq \eval^\text{reg}(Z^N,\Theta), \]
which is the same bound as in \eqref{eq:risk.bound}.  
\end{enumerate} 
Next, it's clear that the common bound derived in the two separate cases above holds uniformly in the actions, i.e., 
\[ \sup_{a \in \AA} \frac{\ell_a(\Theta)}{\uPi_{Z^N}^{\eval \times \rho}(\ell_a)} \leq \eval^\text{reg}(Z^N,\Theta). \]
Plugging $f(\Omega)$ back in for $\Theta$, taking $\uuprob$-expectation on both sides, and then applying the property in Equation~\eqref{eq:reg.ville0}, establishes the bound in Equation~\eqref{eq:dec.bound}. 

\bibliographystyle{apalike}
\bibliography{/Users/rgmarti3/Dropbox/Research/mybib}

\end{document}